\documentclass[11pt]{amsart} 
\usepackage{amssymb,latexsym,amsmath,amsthm}
\usepackage[mathscr]{eucal}
\usepackage{graphicx, color}

\AtBeginDocument{%
   \def\MR#1{}
}

\newtheorem{theorem}{Theorem}[section]
\newtheorem{lemma}[theorem]{Lemma}
\newtheorem{proposition}[theorem]{Proposition}
\newtheorem{corollary}[theorem]{Corollary}
\theoremstyle{definition}
\newtheorem{remark}[theorem]{Remark}

\newtheorem*{notation}{Notation}
\numberwithin{equation}{section}

\renewcommand{\l}{\lambda}

\newcommand{\RR}{\ensuremath{\mathbb{R}}}

\newcommand{\prtl}{\ensuremath{\partial}}

\newcommand{\supp}{\ensuremath{\mathrm{supp}}}
\newcommand{\veps}{\ensuremath{\varepsilon}}

\newcommand{\diag}{\ensuremath{\mathrm{diag}}}

\newcommand{\tubel}{\mathcal{T}_{\l^{-1/2}}(\gamma)}

\newcommand{\tuber}{\mathcal{T}_R}

\title[Logarithmic gain in critical $L^p$ bounds]{Logarithmic improvements in $L^{p}$ bounds for eigenfunctions at the critical exponent in the presence of nonpositive curvature}
\thanks{The first author was supported in part by the National Science Foundation grants DMS-1301717 and DMS-1565436, and the second by the National Science Foundation grant DMS-1665373.} 

\author[M. D. Blair]{Matthew D. Blair}
\address{Department of Mathematics and Statistics, University of New Mexico, Albuquerque, NM 87131, USA}
\email{blair@math.unm.edu}
\author[C. D. Sogge]{Christopher D. Sogge}
\address{Department of Mathematics, Johns Hopkins University, Baltimore, MD 21093, USA}
\email{sogge@jhu.edu}

\begin{document}
\begin{abstract}
  We consider the problem of proving $L^p$ bounds for eigenfunctions of the Laplacian in the high frequency limit in the presence of nonpositive curvature and more generally, manifolds without conjugate points.  In particular, we prove estimates at the ``critical exponent'' $p_c = \frac{2(d+1)}{d-1}$, where a spectrum of scenarios for phase space concentration must be ruled out.  Our work establishes a gain of an inverse power of the logarithm of the frequency in the bounds relative to the classical $L^p$ bounds of the second author.
\end{abstract}
\maketitle
\section{Introduction}\label{S:intro}
Let $(M,g)$ be a boundaryless, compact, connected Riemannian manifold with $d=\dim(M)\geq 2$ and $\Delta_g$ the associated negative Laplace-Beltrami operator.  The spectrum of $-\Delta_g$ is discrete and we let $e_\lambda$ denote any $L^2$-normalized eigenfunction
\begin{equation}\label{efcn}
  (\Delta_g + \lambda^2)e_\lambda =0, \quad \|e_\lambda\|_{L^2(M)}=1.
\end{equation}
Here $L^p(M)$ is the space of $L^p$ functions with respect to Riemannian measure $dV_g$.  The frequency $\lambda$ thus parameterizes the eigenvalues of
\begin{equation*}
P:=\sqrt{-\Delta_g}.
\end{equation*}

We are concerned with $L^p$ bounds on eigenfunctions $e_\lambda$ in the high frequency limit $\lambda\to \infty$ and more generally, ``spectral clusters", meaning sums of eigenfunctions in the range of a spectral projector $\mathbf{1}_{[\lambda,\lambda+h(\lambda)]}(P)$, the operator which projects a function onto all the eigenspaces of $P$ whose corresponding eigenvalue lies in a band of width $h(\lambda)$ to the right of $\lambda$.  In \cite{sogge88}, the second author showed that for $2<p\leq \infty$, with $h(\lambda) \equiv 1$
\begin{gather}
  \left\| \mathbf{1}_{[\lambda,\lambda+1]}(P) \right\|_{L^2(M) \to L^p(M)} \lesssim \lambda^{\delta(p,d)} ,\qquad \lambda \geq 1,\label{sogge88bd}
  \\ \delta(p,d) = \begin{cases}
  \frac{d-1}{2}-\frac{d}{p}, & p_c\leq p \leq \infty,\\
    \frac{d-1}{2}(\frac 12-\frac 1p), & 2 \leq p \leq p_c,
    \end{cases} \qquad p_c := \frac{2(d+1)}{d-1} .\label{soggeexp}
\end{gather}
Note that $\delta(p_c,d) = 1/p_c$.  The case $p=\infty$ here can be seen as a consequence of classical pointwise Weyl laws. One of the key contributions of \cite{sogge88} was to treat these bounds at the ``critical" exponent $p_c$, so that interpolation yields the remaining cases.  This gives a discrete analog of the Stein-Tomas Fourier restriction bound for the sphere \cite[p.135]{soggefica} (or more precisely the adjoint bound).  Given \eqref{sogge88bd}, any eigenfunction as in \eqref{efcn} satisfies
\begin{equation}\label{soggeefcn}
  \|e_\lambda\|_{L^p(M)} \lesssim \lambda^{\delta(p,d)}, \qquad \lambda \geq 1.
\end{equation}

As observed in \cite{sogge86}, the exponent $\delta(p,d)$ in \eqref{soggeefcn} cannot be improved when $(M,g)$ is the round sphere.  The zonal harmonics provide a sequence of eigenfunctions saturating the bound when $p_c \leq p \leq \infty$ and the highest weight spherical harmonics saturate it when $2< p \leq p_c$.  This is not surprising since the spectrum of $P$ in this setting is nearly arithmetic, meaning the projector in \eqref{sogge88bd} is essentially the same as projection onto an eigenspace. However, the geometries for which the corresponding eigenfunctions saturate \eqref{soggeefcn} are in some sense exceptional, and if it does occur then the geodesic flow expects to have similar dynamics to that of the sphere.  Well known classical Gaussian beam constructions show that when $(M,g)$ has a stable elliptic orbit, then there are highly accurate approximate eigenfunctions that saturate \eqref{sogge88bd} when $2<p\leq p_c$.  The works \cite{SoggeZelditchMaximal}, \cite{SoggeTothZelditch}, \cite{SoggeZelditchFocal} characterize geometries which saturate \eqref{soggeefcn} when $p=\infty$, showing, for instance, that in the real analytic case, this will only occur if the unit speed geodesics emanating from a point $x \in M$ loop back to a point at a common time.  These features are absent from several Riemannian manifolds of interest such as  manifolds without conjugate points.

When $(M,g)$ has nonpositive sectional curvatures, it is known that
\begin{equation}\label{loggaininfty}
    \left\| \mathbf{1}_{[\lambda,\lambda+(\log\lambda)^{-1}]}(P) \right\|_{L^2(M) \to L^p(M)} \lesssim_{p} \frac{\lambda^{\delta(p,d)}}{(\log\lambda)^{\sigma(p,d)}}, \qquad p \neq p_c,
\end{equation}
for some exponent $\sigma(p,d)>0$.  For $p_c < p \leq \infty$, a work of Hassell and Tacy \cite{HassellTacyNonpos} shows that one can take $\sigma(p,d) = \frac 12$ here\footnote{The works \cite{Berard}, \cite{bonthonneau2016lower} show this also holds if $(M,g)$ merely lacks conjugate points.}, though the implicit constant tends to infinity as $p \searrow p_c$.  Their work draws from a classical work of Berard \cite{Berard} on the remainder in the pointwise Weyl law (which already implies the $p=\infty$ case).  When $2<p<p_c$, the bounds \eqref{loggaininfty} result from the authors' works \cite{BlairSoggeRefinedHighDimension}, \cite{BlairSoggeToponogov}, but the exponents obtained satisfy $\lim_{p\to p_c-} \sigma(p,d)=0$, again leaving the critical $p=p_c$ case open.  In particular, \cite{BlairSoggeRefinedHighDimension} shows that the left hand side of \eqref{loggaininfty} is dominated by so-called ``Kakeya-Nikodym'' averages, which bound the mass of these spectral clusters within shrinking tubular neighborhoods about a geodesic segment.  The work \cite{BlairSoggeToponogov} then shows that these averages are then seen to exhibit a logarithmic gain in the presence of nonpositive curvature (cf. \eqref{kaknikaverage} below).

The two strategies outlined here are therefore very effective towards obtaining a logarithmic gain in the $L^2(M)\to L^p(M)$ bounds on the projector in \eqref{loggaininfty} when $p \neq p_c$: they either rule out mass concentration similar to the zonal harmonics, yielding improvements for $p_c < p \leq \infty$, or concentration similar to the highest weight spherical harmonics, yielding improvements for $2<p<p_c$.  However, by themselves they do not seem to give an effective strategy for obtaining a logarithmic gain at the critical exponent $p_c$.

We also remark that a work of Hezari and Rivi\`ere \cite{HezariRiviere} uses quantum ergodicity methods to show a logarithmic gain in the $L^{p}$ bounds on eigenfunctions in the presence of negative curvature for all $2<p\leq \infty$, but this is only for a full density subsequence of them.  In the present work, we are interested in bounds satisfied by the full sequence.

A breakthrough on this critical problem came from the second author in \cite{sogge2015improved}, who demonstrated a gain of an inverse power of $\log \log \lambda$ in the $L^2(M)\to L^{p_c}(M)$ bounds on this spectral projector.  The strategy there was to instead consider bounds on the projector in weak-$L^p$ spaces, which in turn yields strong $L^p$ bounds after interpolation with Lorentz space bounds of Bak and Seeger \cite{BakSeegerST}.  The weak bounds were then treated by methods analogous to Bourgain's approach to Fourier restriction to the sphere in \cite[\S6]{BourgainBesicovitch}.  We outline the strategy of \cite{sogge2015improved} in \S\ref{S:reviewandmicro} below.

In the present work, we show that the $\log \log \lambda$ gain exhibited in \cite{sogge2015improved} can be improved to a $\log \lambda$ gain.  This is significant as the latter essentially corresponds to the largest time scale over which the frequency localized wave kernel is currently understood in the setting of nonpositive curvature, closely related to considerations involving the ``Ehrenfest time" in quantum mechanics.  In what follows, $\rho$ is an even, real valued function satisfying
\begin{equation*}
  \rho \in \mathcal{S}(\RR),\; \rho(0)=1, \;\supp(\widehat{\rho}) \subset \left\{|t| \in (1/4,1/2)\right\}.
\end{equation*}
We also assume that $|\rho(t)| \leq 1$ for every $t\in\RR$ so that for any  $\tau,\lambda>0$
\begin{equation}\label{rhoLinfty}
  \left\| \rho(\tau(\lambda - P))\right\|_{L^2(M)\to L^2(M)} \leq 1.
\end{equation}
Throughout the work, we let $c_0>0$ be a sufficiently small but fixed constant and define $T=T(\lambda)$ by
\begin{equation}\label{Tdef}
  T:= c_0 \log \lambda.
\end{equation}

\begin{theorem}\label{T:mainthm}
Suppose $(M,g)$ has nonpositive sectional curvatures. There exists $\veps_0>0$ such that
  \begin{equation}\label{mainbound}
    \| \rho(T(\lambda - P))\|_{L^2(M) \to L^{p_c}(M)} \lesssim \frac{\lambda^{1/p_c}}{(\log\lambda)^{\veps_0}}, \qquad p_c = \frac{2(d+1)}{d-1}.
  \end{equation}
Consequently,
\begin{equation}\label{loggainpc}
    \left\| \mathbf{1}_{[\lambda,\lambda+(\log\lambda)^{-1}]}(P) \right\|_{L^2(M) \to L^{p_c}(M)} \lesssim \frac{\lambda^{1/p_c}}{(\log\lambda)^{\veps_0}}.
\end{equation}
and in particular, any eigenfunction as in \eqref{efcn} satisfies
\begin{equation*}
 \|e_\lambda\|_{L^{p_c}(M)} = \mathcal{O}(\lambda^{1/p_c}(\log\lambda)^{-\veps_0}).
\end{equation*}
\end{theorem}

That \eqref{loggainpc} follows from \eqref{mainbound} is standard. Indeed, taking $c_0$ sufficiently small, $\rho(T(\lambda - P))$ is invertible on the range of $\mathbf{1}_{[\lambda,\lambda+(\log\lambda)^{-1}]}(P)$ with inverse uniformly bounded on $L^2(M)$. We are thus focused on proving \eqref{mainbound}.

\begin{remark}\label{R:optimalexponentintro}
The argument shows that in fact
\begin{equation}\label{eps0}
\veps_0 = \begin{cases}
            \frac{4}{3(d+1)^3}, & d \geq 4, \\
            \frac{1}{48}\,- & d =3\\
            \frac{1}{72}, & d=2,
          \end{cases}
\end{equation}
where in the $d=3$ case the minus sign means that the exponent can be taken strictly less than but arbitrarily close to 1/48.
As noted in Remark \ref{R:negcurvature} below, this exponent can be improved when $d=2,3$ if it is assumed that $(M,g)$ has strictly negative sectional curvatures, leading to an exponent of $\veps_0 = \frac{1}{36}$ when $d=2$ and a removal of the loss when $d=3$ so that $\veps_0 = \frac{1}{48}$.
\end{remark}


To gain an appreciation as to why treating the case of ``critical'' exponents is subtle, it is helpful to consider the analog of \eqref{sogge88bd} for the constant coefficient Laplacian on $\RR^d$, which amounts to considering the Fourier multiplier onto frequencies $\{\xi \in \RR^d: \lambda \leq |\xi| \leq \lambda+1\}$.  Correspondingly, take $f_{\lambda,\theta}\in L^2(\RR^d)$, $\|f_{\lambda,\theta}\|_{L^2}=1$, to be defined as the function whose Fourier transform is the characteristic function of the set
\begin{equation}\label{flambdatheta}
 \left\{ \xi \in \RR^d: \lambda \leq |\xi| \leq \lambda+1, \left| \frac{\xi}{|\xi|}-(1,0,\dots,0)\right| \leq \theta \right\},
\end{equation}
multiplied by an $L^2$-normalization constant $c_{\lambda,\theta}\approx \lambda^{-\frac{d-1}{2}}\theta^{-\frac{d-1}{2}}$. When $\theta = \lambda^{-1/2}$, a linearization of the phase function of the Fourier integral $\int e^{ix\cdot \xi} \hat{f}_{\lambda,\lambda^{-1/2}}(\xi)d\xi$ shows that $|f_{\lambda,\lambda^{-1/2}}(x)|\gtrsim \lambda^{\frac{d-1}{4}}$ on the set
\begin{equation}\label{knappregion}
 \left\{x = (x_1,x') \in \RR\times \RR^{d-1}:  |x'| \leq \veps \lambda^{-1/2}, |x_1| \leq \veps \right\},
\end{equation}
for some $\veps>0$.  It is then easily verified that $\|f_{\lambda,\lambda^{-1/2}}\|_{L^p} \gtrsim \lambda^{\frac{d-1}{2}(\frac 12-\frac 1p)}$, resulting in a function analogous to the ``Knapp example" from Fourier restriction theory.  On the other hand, when $\theta =1$, stationary phase suggests
\begin{equation}\label{slowdecay}
 |f_{\lambda,1}(x)| \approx \lambda^{\frac{d-1}{2}}(1+\lambda|x|)^{-\frac{d-1}{2}}, \qquad |x_1| \gtrsim |x'|,
\end{equation}
for sufficiently many $x$ so that $\|f_{\lambda,1}\|_{L^p} \gtrsim \lambda^{\frac{d-1}{2}-\frac dp}$ when $p>\frac{2d}{d-1}$.  This yields families of functions which saturate the exponent in \eqref{sogge88bd} when $2<p\leq p_c$ and $p_c \leq p \leq \infty$ respectively.  However, by carefully splitting into oscillatory regions where stationary phase can be applied and $\theta$-dependent non-oscillatory regions similar to \eqref{knappregion}, it can be seen that at $p=p_c$, $\|f_{\lambda,\theta}\|_{L^{p_c}} \gtrsim \lambda^{1/p_c}$ for \emph{any} $\theta \in [\lambda^{-1/2}, 1]$, hence its designation as the ``critical'' exponent.  These computations were carried out rigorously in \cite{tacy2016note}.

Analogous constructions can be carried out for suitable approximations to $\mathbf{1}_{[\lambda,\lambda+1]}(P)$ on any $(M,g)$, only now the $x_1$ axis is replaced by a geodesic segment and Riemannian distance replaces Euclidean (see \cite[Ch.5]{soggefica}).  Moreover, localization analogous to that in \eqref{flambdatheta} can be achieved by pseudodifferential operators (PDOs).  These considerations demonstrate that in order to show \eqref{loggainpc} at the critical exponent, one must rule out a spectrum of scenarios for phase space concentration: simply disproving either maximal mass concentration in $\lambda^{-1/2}$ tubular neighborhoods or decay akin to \eqref{slowdecay} as in previous works is not enough by itself.  We shall see that the method in \cite{sogge2015improved} is effective in proving nonconcentration for $\theta \geq \lambda^{-1/2+\veps}$ for any fixed $\veps \in (0,1/2)$.  A key idea in the present work to accomplish this for microlocalized modes corresponding to the remaining cases $\theta \in (\lambda^{-1/2} , \lambda^{-1/2+\veps})$.

Unlike \cite{sogge2015improved}, the present work does not rely on the known bounds \eqref{loggaininfty} when $2<p<p_c$. The bounds in Theorem \ref{T:mainthm} can be interpolated with the $p=2$ case to show $L^p$ bounds for this range of $p$.  As noted above, the exponent $\sigma(p,d)$ vanishes as $p \nearrow p_c$ so the interpolation yields an improved exponent for $p$ interval to the left of $p_c$, but not all values of $2<p<p_c$.

\subsection*{Outline of the work}
In \S\ref{S:reviewandmicro}, we review the method introduced in \cite{sogge2015improved}.  We then show how to generate improvements on this approach for modes microlocalized to a conic sector about a fixed covector field, analogous to the angular localization in \eqref{flambdatheta}.  This adapts the approach in \cite{BlairSoggeToponogov}.

The third section then details a proof by contradiction for our main result, Theorem \ref{T:mainthm}.  The arguments here are partially inspired by strategies in nonlinear PDE, particularly dispersive ones, which seek to characterize the phase space concentration of solutions which develop a singularity, then disprove the possibility of such concentration.  While the present work does not develop an explicit ``profile decomposition" for spectral clusters, akin to those which are common for nonlinear Schr\"odinger equations, the approach here is reminiscent of works in that vein such as \cite{BourgainRefinements}, \cite{BegoutVargas}.  In  \S\ref{SS:sigmaanalysis}, we review the local structure of spectral multipliers which roughly project onto frequency bands of width 1 and then define an almost orthogonal decomposition adapted to these operators which achieves the microlocalization considered in \S\ref{S:reviewandmicro}.  This culminates in the statement of Theorem \ref{T:bubblethm}, which bounds the weak-$L^{p_c}$ quasi-norms of such spectral multipliers by the mass of the elements in the decomposition.  The contradiction is then finalized in \S\ref{SS:finalize}.  The proof of Theorem \ref{T:mainthm} thus relies in a crucial way on Theorem \ref{T:bubblethm} and the improvements from \S\ref{S:reviewandmicro}, in particular Corollary \ref{C:microlp}.  Together these are the central developments in the present work.

The fourth section sets the stage for bilinear estimates on approximate projections onto bands of width 1, which will yield the proof of Theorem \ref{T:bubblethm}.  We then need to show how the elements of our decomposition behave under these approximate projections, which is done in \S\ref{SS:pdobilinear}.  The bilinear estimates can then be concluded.  The final subsection \S\ref{SS:aoforbilinear} then shows that products of the members of the decomposition obey an almost orthogonality principle in $L^r$ spaces, a crucial lemma in the proof of the bilinear bounds.


The fifth and final section then considers results for geometric hypotheses on $(M,g)$ weaker than nonpositive curvature.

\subsection*{Semiclassical analysis}
 This work uses a modest amount of semiclassical analysis, though instead of using the notation $h$ commonly used in this practice, we use $\lambda = h^{-1}$ as the frequency parameter.  The primary use is to quantize various compactly supported pseudodifferential symbols $q_\lambda(x,\xi)$ so that $Q_\lambda=$Op$(q_\lambda)$ is the operator with Schwartz kernel
 \begin{equation}\label{standardquant}
  Q_\lambda(x,y) = \frac{\lambda^{d}}{(2\pi)^d} \int e^{i\lambda(x-y)\cdot \xi} q_\lambda(x,\xi)\,d\xi\quad \text{(standard quantization)}.
 \end{equation}
 In the present work, one will be able to view these operations as the result of taking a classical symbol, compactly supported where $|\xi| \approx \lambda$ with uniform estimates in $S^{0}_{1,0}$, $S^{0}_{7/8,1/8}$, and applying the rescaling $\xi\mapsto \lambda\xi$.  Such a rescaling yields symbols in the classes $S_0,S_{1/8}$ respectively in the sense of \cite[\S4.4]{ZworskiSemiclassicalAnalysis}  The semiclassical Fourier transform is thus defined consistently by $\mathscr{F}_\lambda(f)(\xi) = \hat{f}(\lambda\xi)$ with inverse $\mathscr{F}_\lambda^{-1}(f)(x) = \lambda^{d}\check{f}(\lambda x)$ where $\hat{f}$, $\check{f}$ are the classical Fourier transform and its inverse respectively. The use of semiclassical quantization makes for a convenient use of stationary phase.

\begin{notation}
We take the common convention that $A\lesssim B$ means that $A \leq CB$ for some large constant $C$ which depends only on $(M,g)$ and in particular is uniform in $\lambda$ and possibly other parameters except when they are given in the subscript of $\lesssim$.  Similarly, $A \ll B$ means that $A\leq cB$ for some small uniform constant $c$.   The notation $A\approx B$ means that $A \lesssim B$ and $B\lesssim A$.  Certain variables may be reassigned when the analysis in a given section is independent of prior sections.

Throughout, $\rho_\lambda$ abbreviates the operator $\rho(T(\lambda - P))$ in \eqref{mainbound}, where $T$ is as in \eqref{Tdef}.  We will also use ``local" projectors $\sigma_\lambda$ defined by $\rho(\tilde{c}_0(\lambda-P))$ for some fixed, but sufficiently small constant $\tilde{c}_0$ (much less than the injectivity radius of $(M,g)$).  When these operators are restricted to some sequence of $\lambda_k \to \infty$, we abbreviate $\rho_{\lambda_k}$, $\sigma_{\lambda_k}$ as $\rho_k$, $\sigma_k$ respectively.  Finally, we use $\Theta_g(x,y)$ to denote the Riemannian distance between two points $x,y$ on $M$.
\end{notation}

\subsection*{Acknowledgement} The authors are grateful to the anonymous referee for numerous comments which improved the exposition in this work.

\section{Review of \cite{sogge2015improved} and improved weak bounds for microlocalized modes}\label{S:reviewandmicro}
\subsection{Review of \cite{sogge2015improved}}\label{SS:soggereview} We review the arguments of the second author in \cite{sogge2015improved} used to prove \eqref{mainbound} with $\log\lambda$ replaced by $\log\log\lambda$.  We begin by recalling weak-$L^p$ and Lorentz spaces on $(M,g)$ with respect to Riemannian measure.  The weak-$L^p$ functions are the measurable functions for which the following quasi-norm is finite
\begin{equation*}
  \|f\|_{L^{p,\infty}(M)} = \sup_{\alpha>0} \alpha \left| \left\{ x \in M: |f(x)|>\alpha \right\}\right|^{\frac 1p},
\end{equation*}
where the bars are used denote the Riemannian measure.  The well-known Chebyshev inequality shows that functions in $L^p(M)$ are also in weak-$L^p$ with $\|f\|_{L^{p,\infty}(M)} \leq \|f\|_{L^p(M)}$.  More generally, the Lorentz spaces are a family of interpolation spaces which include both $L^p(M)$ and weak-$L^p$.  They are defined by first considering the distribution function for measurable functions as
\begin{equation*}
  d_f(\alpha) = \left| \left\{ x \in M: |f(x)|>\alpha \right\}\right|
\end{equation*}
then defining $L^{p,q}(M)$ as being the measurable functions for which the following quasi-norm is finite
\begin{equation*}
  \|f\|_{L^{p,q}(M)}:= p^{\frac 1q}
  \left(\int_0^\infty\left[ d_f(s)^{1/p} s\right]^q\frac{ds}{s}\right)^{\frac 1q}, \qquad 0<q<\infty.
\end{equation*}
Lorentz spaces are often equivalently defined using the decreasing rearrangement of $f$.  A well known identity from measure theory shows $\|f\|_{L^{p,p}(M)} = \|f\|_{L^p(M)}$. As suggested by the notation above, when $q=\infty$ the Lorentz space $L^{p,\infty}(M)$ is just the weak-$L^p$ functions.

As observed in \cite[\S4]{sogge2015improved}, an interpolation in Lorentz spaces yields (recalling $\rho_\lambda := \rho(T(\lambda-P))$)
\begin{equation}\label{lorentzinterp}
  \| \rho_\lambda\|_{L^2(M) \to L^{p_c}(M)} \lesssim \| \rho_\lambda\|_{L^2(M) \to L^{p_c,\infty}(M)}^{1-\frac{2}{p_c}} \| \rho_\lambda\|_{L^2(M) \to L^{p_c,2}(M)}^{\frac{2}{p_c}}.
\end{equation}
In \cite[Corollary 1.3]{BakSeegerST}, Bak and Seeger showed $\| \rho_\lambda \|_{L^2 \to L^{p_c,2}} = \mathcal{O}(\lambda^{\frac{1}{p_c}})$. Consequently, it suffices to obtain weak $L^{p_c}$ bounds on $\rho_\lambda$.

We consider a slightly more general setting for the weak bounds, considering instead weak bounds for $Q_\lambda \circ \rho_\lambda$ where $Q_\lambda$ is either the identity or a semiclassical pseudodifferential operator as in \eqref{standardquant} corresponding to a compactly supported symbol $q_\lambda \in S_{1/8}$ in that $|\prtl^\alpha q| \lesssim_\alpha \lambda^{1/8}$.  Note that \cite{sogge2015improved} only considers the case where $Q_\lambda$ is the identity.

Fix a unit vector $f \in L^2(M)$, then consider for $\alpha>0$ and some coordinate system $\Omega \subset M$ the sets $A_\alpha$ defined by
\begin{equation}\label{Aalphadef}
  A_\alpha:=\left\{x \in \Omega: \left|\left((Q_\lambda\circ \rho_\lambda)f\right)(x)\right| > \alpha \right\}, \qquad \|f\|_{L^2(M)}=1.
\end{equation}
Denoting the Riemannian measure of this set as $|A_\alpha|$, we seek a bound
\begin{equation}\label{mainweakbound}
  \alpha|A_\alpha|^{\frac 1{p_c}} \lesssim \lambda^{\frac 1{p_c}} (\log\lambda)^{-\veps_1}, \qquad \veps_1 := \frac{\veps_0 p_c}{p_c-2} = \frac{d+1}{2}\veps_0 .
\end{equation}

We begin by restricting attention to the case
\begin{equation}\label{loweralpha}
\lambda^{\frac{d-1}{4}}(\log\lambda)^{-\frac 12} \lesssim  \alpha.
\end{equation}
We now set
\begin{equation}\label{rchoice}
 r:=\lambda \alpha^{-\frac{4}{d-1}}(\log \lambda)^{-\frac{2}{d-1}}  \qquad \text{ so that } \qquad \left(\lambda r^{-1}\right)^{\frac{d-1}{2}} = \alpha^2\log \lambda.
\end{equation}
Given \eqref{loweralpha}, $r \ll 1$.  At the cost of replacing $A_\alpha$ by a set of proportional measure, we may write $A_\alpha = \cup_j A_{\alpha,j}$ where $d(A_{\alpha,j},A_{\alpha,k}) > C_0 r$ in Euclidean distance for some $C_0>0$ sufficiently large when $j \neq k$.  To see this, cover the original set $A_\alpha$ by a lattice of nonoverlapping cubes of sidelength $r$.  Then partition the cubes in this cover into $\mathcal{O}(1)$ subcollections such that the centers of the cubes in each subcollection are separated by a distance of at least $4^dC_0r$.  By the pigeonhole principle, the intersection of at least one subcollection in the partition with $A_\alpha$ must have measure comparable to $A_\alpha$.  We may thus replace $A_\alpha$ by its intersection with this subcollection of cubes.

Now let $\mathbf{1}_A$ denote the characteristic function of $A$, and $a_j = \mathbf{1}_{A_j} \psi_\lambda$ where $\psi_\lambda$ is defined as
\[
 \psi_\lambda(x) =
 \begin{cases}
 \frac{\left((Q_\lambda\circ \rho_\lambda)f\right)(x)}{|\left((Q_\lambda\circ \rho_\lambda)f\right)(x)|}, & \left((Q_\lambda\circ \rho_\lambda)f\right)(x) \neq 0,\\
 1, & \left((Q_\lambda\circ \rho_\lambda)f\right)(x) = 0.
 \end{cases}
\]
Since $\rho_\lambda$ is self-adjoint and $\|f\|_{L^2(M)}=1$,
\[
 \alpha |A_\alpha| \leq \left| \int \left((Q_\lambda\circ \rho_\lambda)f\right)\overline{\psi_\lambda \mathbf{1}_{A_\alpha}}\right| \leq \bigg(\int \big| \sum_j (\rho_\lambda \circ Q_\lambda^*)a_j\big|^2 \bigg)^{\frac 12}.
\]
This now yields (with $\rho_\lambda^2 = \rho_\lambda \circ \rho_\lambda $)
\begin{equation}\label{oneplustwo}
 \alpha^2 |A_\alpha|^2 \leq \sum_j \int \left| (\rho_\lambda \circ Q_\lambda^*)a_j\right|^2  + \sum_{j\neq k} \int (Q_\lambda \circ \rho_\lambda^2 \circ Q_\lambda^*) a_j \overline{a_k} =: I + II.
\end{equation}
We now consider the consequences of \eqref{oneplustwo} when $Q_\lambda$ is the identity and when this is a semiclassical PDO with $q_\lambda \in S_{1/8}$ separately.

\subsubsection{Consequences of \eqref{oneplustwo} when $Q_\lambda$ is the identity}\label{SSS:qlambdaidentity} We further review the arguments in \cite{sogge2015improved}, assuming $Q_\lambda$ is the identity.  The arguments in \cite{Berard}, \cite{HassellTacyNonpos} used to prove \eqref{loggaininfty} when $p=\infty$ also show that
\begin{equation*}
\|\rho_\lambda\|_{L^2(M) \to L^\infty(M)} \lesssim \lambda^{\frac{d-1}{2}}(\log\lambda)^{-1/2}.
\end{equation*}
In fact, this is a consequence of \eqref{berardbd} below and duality.  Hence $A_\alpha$ as defined in \eqref{Aalphadef} is vacuous unless $\alpha \lesssim \lambda^{\frac{d-1}{2}}(\log\lambda)^{-1/2}$, meaning we only need to consider cases where $r \gtrsim \lambda^{-1}$.

In \cite[(30)]{sogge2015improved}, it is shown that $\rho_\lambda$ satisfies local $L^2$ bounds over balls $B(x,r)$ when $\lambda^{-1} \lesssim r \leq \text{inj}(M)$
\begin{equation}\label{localizedL2}
\|\rho_\lambda \|_{L^2(M)\to L^2(B(x,r))} = \|\rho_\lambda \|_{L^2(B(x,r))\to L^2(M)}   \lesssim r^{\frac 12},
\end{equation}
where the implicit constant is independent of $x$ and the equality holds since $\rho_\lambda$ is self-adjoint.  Hence
\begin{equation}\label{Ibound}
 I \lesssim r \sum_j \int \left| a_j\right|^2 \lesssim r |A_\alpha| = \lambda \alpha^{-\frac{4}{d-1}}(\log \lambda)^{-\frac{2}{d-1}} |A_\alpha|.
\end{equation}

Moreover, with $K(w,z)$ denoting the integral kernel of $\rho_\lambda^2$
\begin{align}
 II &\lesssim \left( \sup_{j\neq k} \sup_{(w,z) \in A_{\alpha,j} \times A_{\alpha,k}} |K(w,z)| \right) \sum_{j\neq k} \|a_j\|_{L^1}\|a_k\|_{L^1} \label{IIboundprep}\\
&\lesssim \left( \sup_{j\neq k} \sup_{(w,z) \in A_{\alpha,j} \times A_{\alpha,k}} |K(w,z)| \right) |A_\alpha|^2.\notag
\end{align}
Lemma 3.3 in \cite{sogge2015improved} then appeals to results of B\'erard \cite{Berard} to observe that there exists $C=C(M,g)$ sufficiently large such that
\begin{equation}\label{berardbd}
 \left|K(w,z)\right| \leq \frac{C}{T}\left(\frac{\lambda}{\lambda^{-1}+\Theta_g(w,z)}\right)^{\frac{d-1}{2}} + C\lambda^{\frac{d-1}{2}}\exp(CT).
\end{equation}
Recalling \eqref{Tdef}, we then have that
\begin{align}
 II &\lesssim \left(C_0^{-\frac{d-1}{2}}(\log\lambda)^{-1}(\lambda r^{-1})^{\frac{d-1}{2}} + \lambda^{Cc_0 + \frac{d-1}{2}}\right)|A_\alpha|^2\notag\\
 &\lesssim  \left(C_0^{-\frac{d-1}{2}}\alpha^2 + \lambda^{Cc_0 + \frac{d-1}{2}}\right)|A_\alpha|^2 \label{IIbound}
\end{align}
Given \eqref{oneplustwo}, \eqref{Ibound}, and \eqref{IIbound} we then have desirable bounds on $|A_\alpha|$ when $\alpha \geq \lambda^{\frac{d-1}{4}+\veps}$ where $\veps$ can be made small by choosing $c_0$ much smaller and $C_0$ large.  However, the smaller we wish to take $\veps$, the smaller we must take $c_0$, which does have to be uniform in the proof.  In \cite{sogge2015improved}, this is remedied by taking $T=c_0\log\log\lambda$ and appealing to the results in \cite{BlairSoggeToponogov}, \cite{BlairSoggeRefinedHighDimension}, to handle smaller values of $\alpha$.  This in turn only yields a gain of a power of $(\log\log\lambda)^{-1}$ in the final estimates.  In the present work, we assume $c_0$ is small enough so that the argument outlined here yields 
\begin{equation}\label{eighththreshold}
 \alpha|A_\alpha|^{\frac{1}{p_c}} \lesssim \lambda^{\frac{1}{p_c}}(\log\lambda)^{-\frac{1}{d+1}} \qquad \text{for } \lambda^{\frac{d-1}{4}+\frac 18} \leq \alpha,
\end{equation}
so that the crucial matter is to treat the cases $\alpha < \lambda^{\frac{d-1}{4}+\frac 18}$.  The choice of $\veps=\frac 18$ is not crucial, but a convenient choice for the sake of concreteness as it does influence other parameters throughout the work.  We stress that in the remainder of this work, \eqref{eighththreshold} is only applied to the case $Q_\lambda = I$.

\subsubsection{Consequences of \eqref{oneplustwo} when $Q_\lambda$ is a semiclassical PDO} We now reconsider the bounds on $I$ and $II$ just established in \S\ref{SSS:qlambdaidentity} but with $Q_\lambda$ now a semiclassical PDO with symbol in $S_{1/8}$.  We would like for \eqref{localizedL2} to yield
\begin{equation}\label{Iboundsemi}
 I \lesssim r \sum_j \int \left| Q_\lambda^* a_j\right|^2 \lesssim r |A_\alpha| = \lambda \alpha^{-\frac{4}{d-1}}(\log \lambda)^{-\frac{2}{d-1}} |A_\alpha|.
\end{equation}
However, the kernel of $Q_\lambda^* $ is only rapidly decreasing outside a $\lambda^{-7/8}$ neighborhood of the diagonal and hence this estimate does not follow at scales finer than $r \leq \lambda^{-7/8}$.  But given \eqref{eighththreshold}, we will only need to bound $I$ when $\alpha < \lambda^{\frac{d-1}{4}+\frac 18}$, meaning that $r > \lambda^{-\frac{1}{2(d-1)}}(\log\lambda)^{-\frac{2}{d-1}}$, which always determines a much coarser scale of at least $r \gg \lambda^{-3/4}$.  Hence in these cases,
$| Q_\lambda^* a_j(x)| = \mathcal{O}(\lambda^{-N})$ for any $N$ outside a cube of sidelength $\approx r$, so the local estimates in \eqref{localizedL2} do indeed yield \eqref{Iboundsemi}.

Turning to the bounds on $II$ in \eqref{IIboundprep}, we now consider the effect of replacing the kernel $K(w,z)$ of $\rho_\lambda^2$ there by the kernel of $Q_{\lambda}\circ \rho_\lambda^2\circ Q_\lambda^*$ as indicated by \eqref{oneplustwo}.  In the next subsection, we will show that for suitable choices of $Q_\lambda$, the corresponding kernel $K(w,z)$ satisfies
\begin{equation*}
\left|K(w,z) \right| \lesssim \frac 1T\left(\frac{\lambda}{\lambda^{-1}+\Theta_g(w,z)}\right)^{\frac{d-1}{2}} + c(\lambda) \lambda^{\frac{d-1}{2}},
\end{equation*}
for some $c(\lambda) \searrow 0$ at least as fast as an inverse power of $\log\lambda$ but no faster than $(\log \lambda)^{-1}$ (so that \eqref{loweralpha} is ultimately respected in this argument). Hence \eqref{IIbound} can be improved to read
\begin{equation}\label{kernelconseq}
II \lesssim \left(C_0^{-\frac{d-1}{2}}\alpha^2 + \lambda^{\frac{d-1}{2}}c(\lambda)\right)|A_\alpha|^2 .
\end{equation}
Taking $C_0$ sufficiently large, we obtain an improvement on \eqref{eighththreshold}:
\begin{equation}\label{weakbdimproved}
\alpha|A_\alpha|^{\frac{1}{p_c}} \lesssim  \lambda^{\frac{1}{p_c}}(\log\lambda)^{-\frac{1}{d+1}} \qquad \text{for } \lambda^{\frac{d-1}{4}}c(\lambda)^{\frac 12} \lesssim \alpha \leq \lambda^{\frac{d-1}{4}+\frac 18}.
\end{equation}

\subsection{Improved weak estimates for microlocalized modes}\label{SS:berardmicro}
Consider any local coordinate chart $\Omega$ on $M$.  Suppose $q_{\lambda}(x,\xi)$ is a semiclassical symbol such that for some unit covector field $\omega(x)$, $|\omega(x)|_{g(x)} = 1$ (with $g(x)$ the ``cometric", the inner product on the $T^*M$ induced by the metric),
\begin{equation}\label{microcutoff}
\begin{gathered}
  \supp(q_{\lambda}) \subset \left\{(x,\xi) \in T^* \Omega: \left|\xi/|\xi|_{g(x)}-\omega(x) \right|_{g(x)} \lesssim \lambda^{-1/8}, |\xi| \approx 1\right\},\\
  \left|\langle \omega(x), d_\xi\rangle^j \prtl^\beta_{x,\xi}  q_{\lambda}(x,\xi)\right|\lesssim_{\beta, j} \lambda^{|\beta|/8}.
  \end{gathered}
\end{equation}
The symbol $q_\lambda$ thus lies in the subcritical class\footnote{Again, the choice of 1/8 is not crucial here, only a convenient one.}  $S_{1/8}$ (as in \cite[\S4.4]{ZworskiSemiclassicalAnalysis}).  If one sets $Q_{\lambda}:=\text{Op}(q_\lambda)$ as in \eqref{standardquant}, we show the following improvement on \eqref{berardbd} of the kernel of the composition $Q_{\lambda}\circ \rho_\lambda^2\circ Q_\lambda^*$:
\begin{theorem}\label{T:berardmicro}
Let $K(w,z)$ denote the kernel of $Q_{\lambda}\circ \rho_\lambda^2\circ Q_\lambda^*$.  We then have
\begin{equation}\label{berardmicro}
  \begin{gathered}\left|K(w,z) \right| \lesssim \frac 1T \left(\frac{\lambda}{\Theta_g(w,z)}\right)^{\frac{d-1}{2}} + c(\lambda) \lambda^{\frac{d-1}{2}},\\
  c(\lambda) =\begin{cases}
  (\log \lambda)^{-1/2}, &\mbox{if } d=2,\\
  (\log \lambda)^{-1}\log \log \lambda, & \mbox{if } d = 3, \\
  (\log \lambda)^{-1}, & \mbox{if } d\geq 4.
\end{cases}
\end{gathered}
\end{equation}
where the implicit constants can be taken independent of $\lambda$ and depend only on finitely many of the derivative bounds in \eqref{microcutoff}.
\end{theorem}

\begin{corollary}\label{C:microlp}  Let $Q_\lambda,c(\lambda)$ be as in Theorem \ref{T:berardmicro}, $A_\alpha$ as in \eqref{Aalphadef}.  Then
  \begin{equation}\label{microlpweak}
  \alpha|A_\alpha|^{\frac{1}{p_c}} \lesssim \lambda^{\frac{1}{p_c}}c(\lambda)^{\frac{1}{d+1}}, \quad 0<\alpha \leq \lambda^{\frac{d-1}{4}+\frac 18}.
  \end{equation}
\end{corollary}
\begin{proof}[Proof of Corollary \ref{C:microlp}]
Given \eqref{weakbdimproved}, it suffices to assume $\alpha \lesssim \lambda^{\frac{d-1}{4}}c(\lambda)^{\frac 12}.$  But since $\|Q_\lambda \circ \rho_\lambda\|_{L^2(M) \to L^2(M)} \lesssim 1$ uniformly, we have $\alpha |A_\alpha|^{\frac 12} \lesssim 1$, hence
\begin{equation*}
\alpha|A_\alpha|^{\frac{1}{p_c}} = \alpha^{1-\frac{2}{p_c}}\left( \alpha |A_\alpha|^{\frac 12}\right)^{\frac{2}{p_c}} \lesssim \lambda^{\frac{1}{p_c}}c(\lambda)^{\frac{1}{d+1}},
\end{equation*}
by the upper bound on $\alpha$.
\end{proof}
\subsubsection{Consequences of the Hadamard parametrix and the proof of Theorem \ref{T:berardmicro}}
Since $\hat{\rho^2} = \hat{\rho} * \hat{\rho}$ is supported in $[-1,1]$, the key to \eqref{berardmicro} is to bound the following integral by the \emph{second term} on the right hand side of \eqref{berardmicro}:
\begin{equation}\label{integratedwave}
 \frac{1}{2\pi T} \int_{-T}^T (1-\beta)(t) \hat{\rho^2}(t/T) e^{i\lambda t} \left(Q_\lambda \circ \cos(t P)\circ Q_\lambda^*\right)(w, z)\,dt.
\end{equation}
where $\beta$ is of sufficiently small compact support and identically one in a neighborhood of 0.  Indeed, without the factor of $1-\beta$ in the integrand, this is the kernel of $Q_{\lambda}\circ \rho_\lambda^2\circ Q_\lambda^*$, up to negligible errors, by Euler's formula.  It is a classical result of H\"ormander \cite{HormanderSpectralFunction} that  if one replaces $1-\beta$ by $\beta$ here, the resulting kernel is bounded by the first term on the right in \eqref{berardmicro}.

Since $(M,g)$ does not have conjugate points, the kernel of $\cos(t P)$ can be analyzed by lifting to the universal cover $(\tilde M, \tilde g)$ where $\tilde g$ is defined by pulling the metric tensor $g$ back via the covering map.  Fix a fundamental domain $D \subset \tilde{M}$ and let $\tilde{w}$, $\tilde z$ denote the unique points in $D$ which project onto $w,z$ in $M$ via the covering map.  Recall that the classical Cartan-Hadamard theorem ensures that $\tilde M$ is diffeomorphic to $\RR^d$ via the exponential map at any point.  Here we take global geodesic coordinates on $\tilde{M}$ via the exponential map at $\tilde w$.  We also assume that the geodesic in $\tilde{M}$ from $\tilde{w}$ with initial covector $\omega(\tilde{w})$ lies along the first coordinate axis and let $\tilde\gamma(t) = (t,0,\dots,0)$ denote this unit speed geodesic.

If $\tilde{P} = \sqrt{-\Delta_{\tilde g}}$, with $\Delta_{\tilde g}$ the Laplacian on $(\tilde M, \tilde g)$, we have
\begin{equation*}
   \cos (t P)(w,z) = \sum_{\alpha \in \Gamma}\cos (t\tilde P)(\tilde w,\alpha(\tilde z))
\end{equation*}
where $\Gamma$ denotes the group of deck transformations which preserve the covering map\footnote{The proof of Theorem \ref{T:berardmicro} is more or less independent of the other sections, so we temporarily reassign $\alpha$ as indexing $\Gamma$ in the interest of consistency with prior works.}.  Note that by finite speed of propagation, we may restrict attention to the $\alpha\in B(\tilde w,T)$.  For $\tilde x\in D$ and $ \tilde y\in \RR^d$, we first concern ourselves with
\begin{equation*}
V(\tilde x, \tilde y)  :=\frac{1}{2\pi T} \int_{-T}^T (1-\beta)(t) \hat{\rho^2}(t/T) e^{i\lambda t}\cos(t \tilde P)\left(\tilde x, \tilde y\right)dt.
\end{equation*}
If we extend the kernel of $Q_\lambda^*$ to be periodic with respect to $\alpha\in\Gamma$, we have (with $d\tilde x$, $d\tilde y$ implicitly the Riemannian measure with respect to $\tilde g$)
\begin{equation}\label{alphasum}
\begin{gathered}
\eqref{integratedwave} = \sum_{\alpha\in \Gamma} U_\alpha(\tilde w, \tilde z),\\
  U_\alpha(\tilde w, \tilde z) := \int_{\alpha(D)} \int_{D} Q_\lambda (\tilde w, \tilde x) V(\tilde x, \tilde y)   Q_\lambda^* (\tilde y, \alpha^{-1}(\tilde z))d\tilde x d \tilde y.
\end{gathered}
\end{equation}

Using the Hadamard parametrix for the wave equation on $(\RR^d,\tilde g)$ and stationary phase (see for example, \cite[Lemma 5.1]{BlairSoggeKaknik}, \cite[\S3]{BlairSoggeToponogov}, \cite[Lemma 3.1]{ChenSogge}), it is known that
\begin{equation}\label{mainasymp}
 V(\tilde x,\tilde y):= \frac{\lambda^{\frac{d-1}{2}}}{T\Theta_{\tilde{g}}(\tilde x,\tilde y)^{\frac{d-1}{2}}} \sum_{\pm}
  e^{\pm i\lambda \Theta_{\tilde g} (\tilde x,\tilde y)}
  a_{\lambda,\pm} (\tilde x,\tilde y)+ R_\lambda(\tilde x, \tilde y).
\end{equation}
Here $a_{\lambda,\pm}, R_\lambda$ vanish for $\Theta_{\tilde g} (\tilde x, \tilde y)\geq T$ by finite speed of propagation and $a_{\lambda,\pm}$ also vanishes if $\Theta_{\tilde g} (\tilde x, \tilde y)$ is sufficiently small since $\beta$ vanishes in a neighborhood of the origin. The remainder can be taken so that $|R_\lambda(\tilde x, \tilde y)| \lesssim \lambda^{-2}$.  Moreover, $a_{\lambda, \pm}$ can be written as
\begin{equation}\label{coeff}
  a_{\lambda,\pm} (\tilde x,\tilde y) = \vartheta(\tilde x, \tilde y)a_{\lambda,\pm,1} \big(\Theta_{\tilde g} (\tilde x,\tilde y)\big)
  + a_{\lambda,\pm,2} (\tilde x,\tilde y),
\end{equation}
where $|\prtl_r^j a_{\lambda,\pm,1}(r)| \lesssim_j r^{-j}$ and there exists $C_d$ so that for $0<|\beta|<16d$,
\begin{equation}\label{asympreg}
|\prtl^\beta_{\tilde x,\tilde y} \Theta_{\tilde{g}}(\tilde x, \tilde y)|, \,\lambda^2|\prtl^\beta_{\tilde x,\tilde y} a_{\lambda,\pm,2} (\tilde x,\tilde y)|, \,|\prtl^\beta_{\tilde x,\tilde y} \vartheta(\tilde x,\tilde y)|  \lesssim \exp(C_d\Theta_{\tilde{g}}(\tilde x, \tilde y)).
\end{equation}
The function $\vartheta(\tilde x ,\tilde y )$ is the leading coefficient in the Hadamard parametrix.  It is characterized by the property that $dV_g = \vartheta^{-2}(\tilde x,\tilde y) d\mathcal{L}$ in normal coordinates at $\tilde x$, with $\mathcal{L}$ denoting Lebesgue measure on $\RR^d$.  Since $(\tilde M, \tilde g)$ has nonpositive sectional curvatures, it is observed in \cite{SoggeZelditchL4} that $\vartheta$ is uniformly bounded as a consequence of the G\"unther comparison theorem.  Moreover, if the curvatures are strictly negative and bounded above by $-\kappa^2$, the same theorem implies $ \vartheta(\tilde{x},\tilde{y}) \lesssim \exp(-\frac{\kappa(d-1)}{2} \Theta_{\tilde g}(\tilde x, \tilde y)) $.

Given the properties of the support of $a_{\pm,\lambda}$ and $R_\lambda$, there are at most $\mathcal{O}(e^{CT})$ nonzero terms in the sum \eqref{alphasum} as a consequence of lattice point counting arguments.  As observed above, $|R_\lambda(\tilde x, \tilde y)| \lesssim \lambda^{-2}$ and hence by Sobolev embedding and $L^2$ bounds on $Q_\lambda$, we may restrict attention to the sum over $\pm $ in \eqref{mainasymp}.

We next observe that in our global coordinate system, we may assume that up to acceptable $\mathcal{O}(\lambda^{-2})$ error, the kernel of $Q_\lambda$ is of the form
\begin{equation*}
\begin{gathered}
\frac{\lambda^{d}}{(2\pi)^{d}} \int e^{i\lambda(\tilde w -\tilde x)\cdot \eta} q_\lambda(\tilde w, \tilde{x}, \eta)\,d\eta,\\
 \supp(q_{\lambda}) \subset \left\{(\tilde w,\eta) \in T^* D, \tilde x \in D: \left|\eta/|\eta|-(1,0,\dots,0)\right| \lesssim \lambda^{-\frac 18}, |\eta| \approx 1\right\}.
\end{gathered}
\end{equation*}
Here we have used a compound symbol, deviating slightly from \eqref{standardquant} to ensure the kernel is supported in $D\times D$.  We may assume the same for the support of the symbol $q_\lambda^*$ of the adjoint. Restricting attention to the main term in \eqref{mainasymp}, $U_\alpha(\tilde w, \tilde z)$ is a sum over $\pm$
\begin{equation}\label{composedkernels}
\begin{gathered}
  \frac{\lambda^{\frac{5d-1}{2}}}{(2\pi)^{2d}T}\int e^{i\lambda\varphi_{\pm}(\tilde w, \tilde x, \tilde y, \tilde z, \eta, \zeta)}
  q_\lambda(\tilde w, \tilde x, \eta) a_{\pm,\lambda}(\tilde x, \tilde y)q_\lambda^*(\tilde y, \alpha^{-1}(\tilde z), \eta)\,d\tilde x d\tilde y d\eta d\zeta,\\
  \varphi_{\pm}(\tilde w, \tilde x, \tilde y, \tilde z, \eta, \zeta):=(\tilde w -\tilde x)\cdot \eta \pm \Theta_{\tilde g}(\tilde x, \tilde y) + (\tilde y -\alpha^{-1}(\tilde z))\cdot \zeta ,
  \end{gathered}
\end{equation}
where as before the domain of integration is $(\tilde x, \tilde y) \in D \times \alpha(D)$.

Applying stationary phase to \eqref{composedkernels} shows that for any $\alpha \in B(w,T)$,
\begin{equation}\label{Ualphabd}
|U_\alpha(\tilde w, \tilde z)| \lesssim \frac{\lambda^{\frac{d-1}{2}}}{T}\left(\vartheta(\tilde w, \alpha(\tilde z)) +\lambda^{-2}\right)\Big(1+\Theta_g(\tilde w, \alpha(\tilde z))\Big)^{-\frac{d-1}{2}}.
\end{equation}
The main idea in the proof of \eqref{berardmicro} is that one can improve upon this bound when $\alpha(D)$ is outside a tubular neighborhood of $\tilde \gamma$.  The proof is similar to that in \cite{BlairSoggeToponogov} where the authors made use of the following consequence of the Toponogov triangle comparison theorem (see \cite[Proposition 2.1]{BlairSoggeToponogov} for further details).
\begin{lemma}\label{L:Toponogov}
Suppose $(\RR^d,\tilde g)$ is the cover of $(M,g)$ given by the exponential map at $w$ and that its sectional curvatures are bounded below by $-1$.  Given $T\gg 1$ and $\theta \ll 1$, let $C(\theta;T)$ denote the set of points in the metric ball of radius $T$ about $w$ such that the geodesic through the point and $w$ forms an angle less than $\theta$ with $\tilde \gamma$.  Fix $R$ sufficiently large. Then if
\begin{equation*}
  \tuber := \{\tilde x \in \RR^d: \Theta_{\tilde{g}} (\tilde x, \tilde \gamma) \leq R  \},
\end{equation*}
we have $C(\theta_T;T) \subset \tuber$ if $\sin(\frac{\theta_T}{2}) = \frac{\sinh(R/2)}{\sinh T}$.
\end{lemma}
Note that we may assume the sectional curvatures of $(M,g)$ and $(\tilde M,\tilde g)$ are bounded below by $-1$ by rescaling the metric in the outset of the proof.

Fix $R= 100\cdot\text{diam}(D)$. Given the lemma, we take $c_0$ in \eqref{Tdef} so that
\begin{gather}
C(\lambda^{-1/16};c_0\log \lambda) = C(\lambda^{-1/16};T)  \subset \tuber, \text{ and hence}\notag\\
  \left|\pm d_{\tilde w}\Theta_{\tilde g}(\tilde w, \tilde y)-(1,0,\dots,0)\right| \gtrsim \lambda^{-\frac{1}{16}}, \qquad \tilde y \notin \tuber.\label{smallcone}
\end{gather}

\subsubsection{Proof of Theorem \ref{T:berardmicro}} As in \cite{BlairSoggeToponogov}, set
\begin{equation*}
  \Gamma_{\tuber} := \{ \alpha \in \Gamma: \alpha(D) \cap \tuber \neq \emptyset \}.
\end{equation*}
The arguments on p. 202 in that work then show that the cardinality of $\{\alpha \in \Gamma_{\tuber}: \Theta_{\tilde g}(\tilde w,\alpha(\tilde z)) \in [2^k,2^{k+1}] \}$ is $\mathcal{O}(2^k)$.  Therefore given \eqref{Ualphabd},
\begin{equation}\label{stabsum}
\sum_{\alpha \in \Gamma_{\tuber}}|U_{\alpha}(\tilde w, \tilde z)|\lesssim \frac{\lambda^{\frac{d-1}{2}}}{T} \sum_{0 \leq k \lesssim \log_2\lambda} 2^{k}2^{-k\frac{d-1}{2}} \lesssim  c(\lambda)\lambda^{\frac{d-1}{2}}.
\end{equation}
Indeed, so geometric summation shows the inequality.

We are now left to show that
\begin{equation}\label{mainmicro}
  \left| U_\alpha(\tilde w, \tilde z) \right| \lesssim 1, \qquad \text{for }\alpha \notin \Gamma_{\tuber}.
\end{equation}
Indeed, if this holds, then given \eqref{Tdef} we have for some uniform constant $C$,
\begin{equation*}
\sum_{\alpha \notin \Gamma_{\tuber}} \left| U_\alpha(\tilde w, \tilde z) \right| \lesssim e^{CT} \lesssim \lambda^{Cc_0} \lesssim c(\lambda)\lambda^{\frac{d-1}{2}},
\end{equation*}
since we take $c_0$ sufficiently small.

Next observe that with $\varphi_\pm$ as in \eqref{composedkernels}
\begin{equation*}
\begin{gathered}
d_{\tilde x} \varphi_\pm =  \pm  d_{\tilde x} \Theta_{\tilde{g}}(\tilde x, \tilde y)  - \eta,\\
d_\eta \varphi_\pm  = \tilde w - \tilde x, \quad d_\zeta \varphi_\pm = \tilde y - \alpha^{-1}(\tilde z).
\end{gathered}
\end{equation*}
Now recall \eqref{asympreg} and the constant $C_d$ there.  If we take $c_0$ small so that $\lambda^{C_dc_0} \ll \lambda^{1/16}$, integration by parts in \eqref{composedkernels} yields
\begin{multline*}
 |U_\alpha(\tilde w, \tilde z)| \lesssim \\
 \sup_{\tilde x, \tilde y, \eta,\pm } \lambda^{\frac{5d-1}{2}} \Big(1+\lambda^{\frac 78} \left| \pm d_{\tilde x} \Theta_{\tilde{g}}(\tilde x, \tilde y) - \eta\right|+ \lambda^{\frac 78}|\tilde y - \alpha^{-1}(\tilde z)| + \lambda^{\frac 78}|\tilde w - \tilde x| \Big)^{-8d},
\end{multline*}
where the supremum is over all points inside the support of the amplitude.  However, there exists $C$ such that
\begin{equation*}
\left| d_{\tilde x} \Theta_{\tilde{g}}(\tilde x, \tilde y)  - d_{\tilde w}\Theta_{\tilde{g}}(\tilde w, \alpha^{-1}(\tilde z)) \right|
\lesssim e^{CT}\left(|\tilde w-\tilde x|+ |\tilde y-\alpha^{-1}(\tilde z)| \right),
\end{equation*}
so taking $c_0 < \frac{1}{16C}$ in \eqref{Tdef}, the constant on the right is $\lambda^{Cc_0}\ll \lambda^{1/16}$, hence
\begin{equation*}
  |U_\alpha(\tilde w, \tilde z)| \lesssim
 \sup_{\eta,\pm} \lambda^{\frac{5d-1}{2}} \left(1+\lambda^{\frac 34} |  \pm d_{\tilde w} \Theta_{\tilde{g}}(\tilde w, \alpha^{-1}(\tilde z)) - \eta| \right)^{-8d},
\end{equation*}
But since $|\eta -(1,0,\dots,0)| \lesssim \lambda^{-1/8}$, and $\alpha \notin \Gamma_{\tuber}$,  as a consequence of \eqref{smallcone} the second factor is $\mathcal{O}(\lambda^{-3d})$ which is stronger than \eqref{mainmicro}.

\begin{remark}\label{R:negcurvature}
  When the curvatures of $(M,g)$ are strictly negative, one can take $c(\lambda) = (\log \lambda)^{-1}$ in Theorem \ref{T:berardmicro} and its corollary in any dimension, leading to an improvement in the exponent $\veps_0$ in Remark \ref{R:optimalexponentintro} when $d=2,3$ via the argument in \S\ref{S:proofbycontra}.  As observed above, $\vartheta$ decays exponentially in $\Theta_{\tilde{g}}$ in this case, and hence the sum in \eqref{stabsum} is $\mathcal{O}(\lambda^{\frac{d-1}{2}}/\log\lambda)$ for any $d \geq 2$.
\end{remark}

\section{The proof by contradiction}\label{S:proofbycontra}
To obtain a contradiction to Theorem \ref{T:mainthm}, suppose there exists a sequence of triples $\{(f_k,\lambda_k,B_k)\}_{k=1}^\infty$ such that $\|f_k\|_{L^2(M)} =1$, $B_k,\lambda_k \to \infty$ such that
\begin{equation}\label{firstwksaturation}
0<\frac{B_k\lambda_k^{1/p_c}}{(\log\lambda_k)^{\veps_1}}    < \| \rho_k f_k\|_{L^{p_c,\infty}(M)}, \qquad \veps_1 = \frac{\veps_0p_c}{p_c-2} = \frac{d+1}{2}\veps_0,
\end{equation}
where $\veps_0$ is in our main $L^{p_c}$ estimate in Theorem \ref{T:mainthm} (cf. Remark \ref{R:optimalexponentintro}) and as before, $\rho_k = \rho_{\lambda_k}$. Indeed, if we had
\begin{equation*}
\limsup_{\lambda\to\infty}\lambda^{-1/p_c}(\log\lambda)^{\veps_0}\|\rho_\lambda\|_{L^2(M)\to L^{p_c}(M)} = \infty,
\end{equation*}
then a similar inequality holds with different values of $B_k \to \infty$, a strong $L^p$ bound replacing this weak one, and the larger log-exponent $\veps_1$ replaced by $\veps_0$.  But then the Lorentz interpolation argument \eqref{lorentzinterp} yields \eqref{firstwksaturation}.

Taking $\veps_0$ small enough so that $\veps_1 \leq \frac{1}{d+1}$, given the consequence \eqref{eighththreshold} of the results in \cite{sogge2015improved}, we may assume for each $k$, there is $\alpha_k >0$ such that
\begin{equation}\label{initwkblowup}
\frac{B_k\lambda_k^{1/p_c}}{(\log\lambda_k)^{\veps_1}}   < \alpha_k\left|\left\{x\in M: |\rho_k f(x)| >\alpha_k \right\}\right|^{\frac{1}{p_c}}, \qquad
\alpha_k \leq \lambda_k^{\frac{d-1}{4}+\frac 18}.
\end{equation}

In order to take advantage of the improved microlocalized bounds in Theorem \ref{T:berardmicro} and Corollary \ref{C:microlp}, we will appeal to methods emanating from the Fourier restriction problem and their relatives.  In particular, we want to control the $L^{p_c,\infty}$ quasi-norm of the $\rho_k f$ by the $L^{p_c}$ and $L^2$ norm of expressions such as $Q_\lambda \rho_k f$ with $Q_\lambda$ being the pseudodifferential cutoff function as in Theorem \ref{T:berardmicro} (though the notation will change slightly below).  While the operator $\rho_\lambda$ is still too poorly understood to apply such classical methods, we can instead use local operators $\sigma_k=\rho(\tilde{c}_0(\lambda_k-P))$ (as in the notation section) in order to achieve this.  This is in the same spirit of the authors' previous work, and that of others, where the local operators are treated in a way that make them amenable to global analysis.

The main idea is that $(I-\sigma_k)\circ \rho_k$ is an acceptable error term.  Indeed, since $(1-\rho)(0)=0$, we have
\[
|(1-\rho)(\tilde{c}_0(\lambda-\tau))\rho(T(\lambda-\tau))| \lesssim T^{-1}(1+T|\lambda-\tau|)^{-N},
\]
and hence the classical $L^2\to L^{p_c}$ bounds \eqref{sogge88bd} for spectral projectors $\mathbf{1}_{[l,l+1]}(P)$ imply
\begin{equation}\label{sigmareplace}
\|(I-\sigma_k)\circ \rho_k f_k\|_{L^{p_c}(M)} \lesssim  \lambda^{1/p_c}(\log\lambda)^{-1}.
\end{equation}
Since $\veps_1<1$, we may assume that \eqref{initwkblowup} holds with $\sigma_k\rho_k f_k$ replacing $\rho_k f_k$ (the former abbreviating $(\sigma_k\circ \rho_k) f_k$).

Now take a finite partition of unity subordinate to an open cover of a suitable family of coordinate domains.  By the pigeonhole principle, we may assume that at the cost of shrinking the $B_k$ and $\alpha_k$ by a uniform factor and passing to a subsequence of the triples indexed by $k$ there is a bump function $\psi$ supported in a coordinate chart $\Omega\subset\RR^d$ centered at the origin for which
\begin{equation*}
\frac{B_k\lambda_k^{1/p_c}}{(\log\lambda_k)^{\veps_1}}   < \alpha_k \left|\left\{x\in \Omega: |\psi(x)(\sigma_k\rho_k f_k)(x)| >\alpha_k \right\}\right|^{\frac{1}{p_c}}.
\end{equation*}
After another harmless shrinking of $B_k, \alpha_k$, we may also assume that the measure here is the usual Lebesgue measure in coordinates instead of Riemannian measure.  By a second application of the pigeonhole principle, we may assume that there exists a Fourier multiplier $m\in S_{1,0}^0$ truncating to a conic sector of small aperture about a fixed vector such that
\begin{equation}\label{wksaturation}
\frac{B_k\lambda_k^{1/p_c}}{(\log\lambda_k)^{\veps_1}}   <  \alpha_k\left|\left\{x\in \Omega: |(m(D)\psi\sigma_k\rho_k f_k)(x)| >\alpha_k \right\}\right|^{\frac{1}{p_c}}.
\end{equation}
After a possible rotation of coordinates, we may further assume that the fixed vector is $(1,0,\dots,0)$, that is,
\begin{equation*}
 \supp(m) \subset \left\{\xi:\left|\xi/|\xi|-(1,0,\dots,0)\right| \ll 1\right\}.
\end{equation*}

\subsection{Analysis of $\sigma_\lambda$}\label{SS:sigmaanalysis}
We may assume that in the coordinate chart $\Omega$, $g^{ij}(0) = \delta^{ij}$ and that for some $\epsilon>0$ sufficiently small
\begin{equation}\label{epsilonchart}
\Omega = [-\epsilon,\epsilon]^d
\end{equation}

We now recall the method for computing the kernel of $\sigma_\lambda =\rho(\tilde{c}_0(\lambda -P))$ from \cite[Ch. 5]{soggefica}. There it is observed that $\sigma_\lambda$ can be realized as an operator valued integral involving the wave kernel $e^{-itP}$
\[
\sigma_\lambda  = \frac{1}{2\pi \tilde{c}_0}\int_{-\tilde{c}_0}^{\tilde{c}_0} e^{it\lambda}e^{-itP}\widehat{\rho}(t/\tilde{c}_0)\,dt.
\]
Using a Lax parametrix, it is well known that for $|t| \leq \tilde{c}_0$ there exists a phase function $\varphi(t,x,\xi)$ and an amplitude $v(t,x,\xi)$ such that the Schwartz kernel of $m(D)\psi e^{-itP}$ is given by an oscillatory integral
\[
\left(m(D)\psi e^{-itP}\right)(x,y) = 2\pi\tilde{c}_0 \int e^{i(\varphi(t,x,\xi)-y\cdot\xi)} v(t,x,\xi)\tilde{\psi}(y)\,d\xi + \text{error}
\]
where the error is smoothing to a sufficient order and hence can be neglected in what follows.  Here $\tilde{\psi}$ is a bump function of slightly larger support and we may assume $v(t,\cdot,\xi)$, $\tilde{\psi}$ are supported in the same coordinate chart $\Omega$ as above.  Moreover, we may take
\begin{equation*}
  \supp(v(t,x,\cdot)) \subset \{\xi:|\xi/|\xi|-(1,0,\dots,0)| \ll 1\},
\end{equation*}
for some conic sector of slightly larger aperture than the one containing $\supp(m)$ (cf. \eqref{wksaturation}).  Up to negligible error, the kernel of $m(D)\psi\sigma_\lambda$ is
\[
\int_{-\tilde{c}_0}^{\tilde{c}_0} \int e^{i(\lambda t + \varphi(t,x,\xi)-y\cdot\xi)} \widehat{\rho}(t/\tilde{c}_0)v(t,x,\xi)\,d\xi\,dt\cdot\tilde{\psi}(y)
\]
An integration by parts in $t$ shows that the contribution of the region where $|\xi|\ll \lambda$ or $|\xi|\gg \lambda$ to this integral is $\mathcal{O}(\lambda^{-N})$ for any $N$ and hence negligible.  Hence we may assume that $v(t,x,\cdot)$ is further supported where $|\xi|\approx \lambda$.  Rescaling $\xi \mapsto \lambda \xi$, we are reduced to considering a semiclassical Fourier integral operator $\tilde{\sigma}_\lambda$ given by integration against the kernel
\begin{equation}\label{localkernel}
\tilde{\sigma}_\lambda(x,y) := \lambda^d\int_{-\tilde{c}_0}^{\tilde{c}_0} \int e^{i\lambda(t+\varphi(t,x,\xi)-y\cdot\xi)} \widehat{\rho}(t/\tilde{c}_0)v(t,x,\xi)\,d\xi\,dt\cdot\tilde{\psi}(y)
\end{equation}
where now $v(t,x,\cdot)$ is supported where $|\xi|\approx 1$ and in the same conic region as before.  Therefore in what follows, we may assume that any function on which $\tilde{\sigma}_\lambda$ operates has its semiclassical Fourier transform supported in this region.  Note that the operator $\tilde\sigma_\lambda$ is $m(D)\psi\sigma_\lambda$ up to negligible error.

We pause to remark that \cite[Lemma 5.1.3]{soggefica} uses stationary phase on \eqref{localkernel} to show that $\tilde\sigma_\lambda$ is an oscillatory integral operator with Carleson-Sj\"olin phase (see also the $2^j\approx 1$ case of Lemma \ref{L:pdofiocomp} below).  As observed there, the $L^p$ theory for such operators due to H\"ormander and Stein then yield the following linear estimates on $\tilde\sigma_\lambda$, which in turn imply \eqref{sogge88bd}:
\begin{equation}\label{linearestimates}
\|\tilde\sigma_\lambda\|_{L^2 \to L^{p_c}} \lesssim \lambda^{\frac{1}{p_c}}.
\end{equation}

We now want to decompose the identity into a family of pseudodifferential operators which have the effect of localizing a function in phase space in a fashion similar to Fourier multipliers defined by the characteristic functions in \eqref{flambdatheta}.   However, this requires care as the operators must in some sense be invariant under the geodesic flow.  We achieve this by fixing a hyperplane, namely the $x_1 =0$ plane, then localizing the momenta so that it is within a $\lambda^{-1/8}$ neighborhood of a fixed vector as it passes through this hyperplane.  In the construction, it is convenient to use the trivialization $T^*\Omega \cong \Omega \times \RR^d$ to define a family of constant covector fields along the hyperplane which serve as the centers of these neighborhoods (constant in the sense that their expression in the coordinate frame is independent of position).


\subsubsection{Analysis of the geodesic flow} In preparation for the decomposition, we study $\chi_t$, which we denote as the flow on $T^*\Omega$ generated by the Hamiltonian vector field of $p(x,\xi) = |\xi|_{g(x)}$.  Hence $\chi_t(x,\xi)$ is the time $t$ value of the integral curve of the Hamiltonian vector field of $p$ with initial data $(x,\xi)$. Recall that the phase function $\varphi$ in the construction above satisfies
\begin{equation}\label{generatingfcn}
 \chi_t(d_\xi \varphi(t,x,\xi), \xi) = (x,d_x\varphi(t,x,\xi)).
\end{equation}
For initial data in the cosphere bundle $S^*\Omega$ defined by
\begin{equation*}
S^*\Omega := \{(x,\xi) \in T^*\Omega: |\xi|_{g(x)}=1\},
\end{equation*}
the integral curves of $p$ coincide with geodesics of $(M,g)$ as curves in the cotangent bundle.  We write $x=(x_1,x')$ so that in particular $(0,x')$ gives coordinates on the $x_1=0$ hyperplane.  Consider the restriction of this flow to a neighborhood of origin in the hyperplane $x_1=0$ and $\xi$ in a conic neighborhood of $(1,0,\dots,0)$ in $S_x^*\Omega$, the cosphere space at $x$.  Assuming that $\epsilon$ in \eqref{epsilonchart}, and $\tilde{c}_0$ is sufficiently small,
we have for $|t| \leq \tilde{c}_0$, the mapping
\begin{equation*}
  (t,x',\eta) \mapsto \chi_{t}(0,x',\eta)
\end{equation*}
generates a diffeomorphism from the neighborhood to a conic neighborhood of $(1,0,\dots,0)$ in the cosphere bundle $S^*\Omega$.  Indeed, recalling our assumption that $g^{ij}(0) = \delta^{ij}$, the derivative of this mapping at $(0,0,(1,0,\dots,0))$ is the identity.  Denote the inverse as
\begin{equation}\label{planeinverse}
  (\iota(x,\omega),\Phi(x,\omega),\Psi(x,\omega)) \in (-\tilde{c}_0,\tilde{c}_0) \times \{y_1=0\} \times S_{\Phi(x,\omega)}^* \Omega.
\end{equation}
Equivalently, these functions can be described in terms of the minimizing unit speed geodesic passing through $(x,\omega)$:  this geodesic passes through the $y_1=0$ plane at the point $y'=\Phi(x,\omega)$, the covector at this intersection point is given by $\Psi(x,\omega)$, and $\iota(x,\omega) = \Theta_g (x, \Phi(x,\omega))$.

We note that we may further assume that for any $x\in\Omega$, $\omega\mapsto \Psi(x,\omega)$ is an invertible mapping, and if $\eta\mapsto\omega(x,\eta)$ denotes the inverse, then $\omega(x,\eta)$ is the unit covector along the geodesic through $x$ whose covector at the intersection point with $y_1=0$ is $\eta$.

\subsubsection{The almost orthogonal decomposition.} Now let $\nu$ index a collection of vectors in a neighborhood of $(1,0,\dots,0)$ on $\mathbb{S}^{d-1}$ separated by a distance of at least $\frac 12 \lambda^{-1/8}$.   Define a corresponding partition of unity $\beta_\nu(\xi)$ such that $\supp(\beta_\nu)$ is contained in a spherical cap of diameter $2\lambda^{-1/8}$ about $\nu$ and $\sum_{\nu} \beta_\nu(\xi) = 1$ for $\xi \in \mathbb{S}^{d-1}$.  Then extend $\beta_\nu(\xi)$ to all of $\RR^{d}\setminus\{0\}$, so that it is homogeneous of degree zero.  Now define
\begin{equation*}
  q_\nu(x,\xi) = \tilde{\tilde{\psi}}(x)\beta_\nu\big(\Psi(x,\xi/|\xi|_{g(x)})\big)\tilde{\beta}(|\xi|_{g(x)}),
\end{equation*}
where $\tilde\beta$ is a bump function such that $\tilde\psi(x) v(t,x,\xi) = \tilde\psi(x) \tilde{\beta}(|\xi|_{g(x)})v(t,x,\xi)$ is supported where $|\xi| \approx 1$ and in a slightly larger conic region than $v(t,x,\cdot)$.  Moreover, we take $\tilde{\tilde{\psi}}$ to be a bump function supported in $\Omega$ and identically one on $\tilde\psi$.  This bump function means that $q_\nu(x,\xi)$ is not invariant under $\chi_t$, but we can assume that $\tilde{c}_0$ and the support is chosen suitably so that $q_\nu(\chi_t(x,\xi))=q_\nu(x,\xi)$ when $\tilde{\psi}(x)v(t,x,\xi)\neq 0$ and $|t|\leq\tilde{c}_0$.

The function $q_\nu$ thus defines a semiclassical symbol in the class $S_{1/8}$.  It is of the form considered in Theorem \ref{T:berardmicro} where the unit covector field $\omega(x)=\omega(x,\nu)$ is that of the minimizing geodesic passing through $x$ such that its intersection with $y_1=0$ has the covector $\nu/|\nu|_{g(x)}$.
We define $Q_\nu=\text{Op}(q_\nu)$ as in \eqref{standardquant} and hence up to error which is $\mathcal{O}(\lambda^{-N})$ in $L^2$ for some $N$ sufficiently large
\begin{equation}\label{sigmadecomp}
  \tilde\sigma_\lambda h = \sum_\nu \tilde\sigma_\lambda Q_\nu h.
\end{equation}
Moreover, the selection of the indices $\nu$ ensures that there exists a constant $C_d$ such that for any fixed $\nu$
\begin{equation}\label{cdconstant}
  \#\{\tilde{\nu}: \supp(q_\nu) \cap \supp(q_{\tilde{\nu}}) \neq \emptyset \} \leq C_d.
\end{equation}

In this work we will exploit the almost orthogonality of the decomposition \eqref{sigmadecomp} at the level of $L^2$ and also for products of these members in $L^r$ for $1\leq r \leq \infty$.  We begin by considering the former; the more general theory will be considered in Lemma \ref{L:aobound} and is adapted to $\tilde\sigma_\lambda$.

We first observe that by appealing to the FBI transform as in \cite[Theorem 13.3]{ZworskiSemiclassicalAnalysis}, we have for any symbol\footnote{This theorem can be applied to the rescaled symbol $q(\lambda^{-1/8}(x,\xi))$, which yields the decay rate of $\lambda^{-3/4}$ for the error term upon return to the original coordinates.  Since we are working in a subcritical symbol class, the distinction between Weyl quantization and the standard one in \eqref{standardquant} is inconsequential given change of quantization formulae.} $q \in S_{1/8}$
\begin{equation}\label{FBIconseq}
  \| \text{Op}(q) \|_{L^2(\RR^d) \to L^2(\RR^d)} \leq \|q\|_\infty + \mathcal{O}(\lambda^{-\frac 34}).
\end{equation}
Since we may restrict attention to sufficiently large $\lambda$, we have that for any subcollection $F$ of the $\nu$ and an arbitrary sequence $\epsilon_\nu = \pm 1$,
\begin{equation}\label{L2bounds}
\left\|\sum_{\nu\in F}\epsilon_\nu Q_{\nu}\right\|_{L^2 \to L^2},\; \left\|I-\sum_{\nu\in F}Q_{\nu}\right\|_{L^2 \to L^2}, \; \|Q_\nu\|_{L^2 \to L^2} \leq 2.
\end{equation}
\begin{proposition}\label{prop:aoL2}
  Suppose $h \in L^2(\RR^d)$ and the semiclassical Fourier transform of $h$ is supported where $\tilde\beta =1$.  Then for $\lambda$ sufficiently large,
\begin{equation}\label{aoproperty}
(4C_d)^{-1}\|h\|_{L^2}^2 \leq \sum_{\nu} \|Q_\nu h\|_{L^2}^2 \leq 4 \|h\|_{L^2}^2.
\end{equation}
Moreover, if $F$ is any subcollection of the $\nu$,
\begin{equation}\label{aoproperty2}
  \sum_{\nu\in F} \|Q_\nu h\|_{L^2}^2 +\left\|h-\sum_{\nu\in F} Q_\nu h\right\|_{L^2}^2 \leq 4 \|h\|_{L^2}^2.
\end{equation}
\end{proposition}

\begin{proof}
We begin with the first inequality in \eqref{aoproperty}.  The symbolic calculus means that if $\supp(q_\nu) \cap \supp(q_{\tilde{\nu}}) = \emptyset$, then
\begin{equation}\label{offdiagqnu}
  \|Q_{\tilde{\nu}}^* \circ Q_\nu\|_{L^2 \to L^2} \lesssim_N \lambda^{-N}.
\end{equation}
Taking $N>\frac{d-1}{4}$ here we see that for $\lambda$ sufficiently large
\begin{equation*}
  \|h\|_{L^2}^2 \leq 2 \sum \left\{ \langle Q_\nu h, Q_{\tilde\nu} h\rangle_{L^2}: \supp(q_\nu) \cap \supp(q_{\tilde{\nu}}) \neq \emptyset \right\}.
\end{equation*}
Indeed, \eqref{offdiagqnu} ensures that the contribution of the remaining terms is negligible.  An application of Cauchy-Schwarz now show that the right hand side here is in turn bounded by $4C_d\sum_{\nu} \|Q_\nu h\|_{L^2}^2$, where $C_d$ is as in \eqref{cdconstant}.

Turning to the second inequality in \eqref{aoproperty}, this is essentially a consequence of \eqref{L2bounds} and the fact that the constant in Khintchine's inequality can be taken to be 1 when $p=2$.  More directly, consider the usual family of Rademacher functions $r_k(t) = \text{sgn}(\sin(2^k\pi t))$, $k=1,2,3,\dots$, which are known to form an orthonormal sequence in $L^2([0,1])$.  Hence for any injection $\nu \mapsto k(\nu)$, \eqref{L2bounds} gives that
\begin{equation*}
  \sum_\nu \|Q_\nu h\|_{L^2}^2 =  \int_{0}^{1}\int_{\RR^d} \left|\sum_\nu r_{k(\nu)}(t) Q_\nu h(x) \right|^2\,dxdt \leq 4 \|h\|_{L^2}^2.
\end{equation*}
To see \eqref{aoproperty2}, now consider an injection $\nu \mapsto k(\nu)$ defined on $F$ such that $k(\nu) \neq 1$ for all $\nu \in F$.  The triangle inequality implies that for all $t$,
\begin{equation*}
  \left|\sum_{\nu \in F}r_{k(\nu)}(t)q_\nu(x,\xi) + r_1(t)\left(1-\sum_{\nu\in F} q_\nu(x,\xi)\right)\right| \leq 1.
\end{equation*}
Proceeding similarly, we now have
\begin{multline*}
  \sum_{\nu\in F} \|Q_\nu h\|_{L^2}^2 +\left\|h-\sum_{\nu\in F} Q_\nu h\right\|_{L^2}^2 =\\  \int_{0}^{1}\int_{\RR^d} \left|\sum_{\nu\in F} r_{k(\nu)}(t) Q_\nu h(x) + r_1(t)(I-\sum_{\nu\in F} Q_\nu)h(x) \right|^2\,dxdt \leq 4 \|h\|_{L^2}^2.
\end{multline*}
\end{proof}

Note that \eqref{aoproperty} implies the crude bound
\begin{equation}\label{crudebound}
 \|h\|_{L^2} \lesssim \lambda^{\frac{d-1}{16}} \sup_\nu \|Q_\nu h\|_{L^2}.
\end{equation}

In \S\ref{S:aobilinear}, we will prove the following theorem:
\begin{theorem}\label{T:bubblethm}
Suppose $h$ is supported in $\Omega$ with $\|h\|_{L^2(\Omega)} \leq 4$.  Assume further that $h$ satisfies \eqref{aoproperty}, \eqref{crudebound}, and \eqref{sigmadecomp} (without error term).  Let $\tilde{\sigma}_\lambda$ be as in \eqref{localkernel} and define
\begin{equation*}
\tilde{A}_\alpha = \{x \in \Omega:|\tilde{\sigma}_\lambda h(x)| > \alpha\}.
\end{equation*}
Let $\delta_d = 2/(d+1)$ when $d \geq 3$ and $\delta_2 = 1/3$ when $d=2$.  Then
\begin{equation}\label{bubblebd}
\sup\left\{\alpha | \tilde{A}_\alpha |^{\frac{1}{p_c}}: \alpha\in (0,\lambda^{\frac{d-1}{4}+\frac 18}]\right\}
\lesssim \lambda^{\frac{1}{p_c}}\left( \max_\nu \|Q_\nu h \|_{L^2}^{\delta_d} \right) +
\lambda^{\frac{1}{p_c}-}.
\end{equation}
Here the second term on the right means $\lambda$ is raised to some given power which is strictly less than $1/p_c$.
\end{theorem}
As we shall see in \eqref{crucialoffbd} below, the assumption $\alpha\in (0,\lambda^{\frac{d-1}{4}+\frac 18}]$ will allow us to exploit gains in bilinear estimates that correspond to the ``subcritical'' range of $L^q$ spaces with $2 < q < p_c$.  Indeed, \eqref{crucialoffbd} is a subtle but crucial observation in the present work, showing that weak bounds can be combined with known bilinear estimates to avoid the impediments presented by localizing the momenta in scales as fine as $\lambda^{-1/2}$ (a necessary technical difficulty in \cite{BlairSoggeRefinedHighDimension}). The second term in the right side of (3.18) corresponds to the gain in the bilinear estimates corresponding to angular separation larger than $\lambda^{-1/8}$ and the above assumptions on $\alpha$, while the first term in the right side of (3.18) corresponds to the contribution to the bilinear estimate for near-diagonal terms corresponding to separation smaller than $\lambda^{-1/8}$. As noted above, there is nothing special about the power $1/8$. Any number between $0$ and $1/2$ should work after adjusting the power of $\alpha$ correspondingly.


\subsection{Finalizing the contradiction}\label{SS:finalize}
Recall from \eqref{wksaturation}, we have sequences $B_k,\lambda_k \to \infty$ and corresponding $\alpha_k$ satisfying
\begin{equation*}
\frac{B_k\lambda_k^{1/p_c}}{(\log\lambda_k)^{\veps_1}}   <
\alpha_k\left|\left\{ |\tilde{\sigma}_k\rho_k f_k| >\alpha_k \right\}\right|^{\frac{1}{p_c}},\qquad 0<\alpha_k \leq \lambda_k^{\frac{d-1}{4}+\frac 18}.
\end{equation*}
Here we use the same notation convention $\tilde{\sigma}_k= \tilde{\sigma}_{\lambda_k}$, and it is understood that the set on the right in the inequality is $\{x\in \Omega: |(\tilde{\sigma}_k\rho_k f_k)(x)| >\alpha_k \}$.  Recall that the semiclassical wave front set of the kernel of $\tilde\sigma_k$ is contained in $\Omega \times \{ |\xi|\approx 1\}$ given the localization of the symbol $v(t,x,\xi)$ above.  We therefore make a slight abuse of notation and assume that $\rho_k f_k$ satisfies the assumptions on $h$ in Theorem \ref{T:bubblethm}, including having support in $\Omega$, though strictly speaking this only applies to a microlocalization of this function.  In particular we assume $\rho_kf_k$ satisfies \eqref{aoproperty}, \eqref{crudebound}, and \eqref{sigmadecomp} (at the cost of shrinking the $B_k$ and $\alpha_k$ one last time).  By \eqref{aoproperty2}, for any set of $\{Q_{\nu_l}\}_{l=1}^L$ with $\nu_l \neq \nu_j$ when $j\neq l$,
\begin{equation}\label{aocollection}
  \left\|\rho_{k}f_k-\sum_{l=1}^{L}Q_{\nu_l}\rho_kf_k\right\|_{L^{2}(\Omega)}^2 \leq 4\|\rho_kf_k\|_{L^2(\Omega)}^2 -\sum_{l=1}^{L} \|
 Q_{\nu_l}\rho_{k}f_k\|_{L^2(\Omega)}^2
\end{equation}
Let $C(2C_d)^{-\delta_d}$ exceed the implicit constant in \eqref{bubblebd}, where $C_d$ is defined in \eqref{cdconstant}. Take $N_k\in \mathbb{N}$ such that
\[
  \frac{N_k}2  \leq 4\left( \frac{2C(\log\lambda_k)^{\veps_1}}{B_k}\right)^{2/\delta_d} < N_k .
\]
If the middle expression is strictly less than $1/2$, take $N_k=1$.  We note for future use that in either case, we have
\begin{equation*}
  N_k = o((\log\lambda_k)^{2\veps_1/\delta_d}).
\end{equation*}

We claim there exists a selection of distinct $Q_{\nu_1}\rho_kf_k, \dots, Q_{\nu_{N_k}}\rho_kf_k$, with $Q_\nu$ as in Theorem \ref{T:bubblethm}, which satisfies for $k$ sufficiently large
\begin{gather}
\left(\frac{B_k}{2C(\log\lambda_k)^{\veps_1}} \right)^{1/\delta_d} \leq \left\|Q_{\nu_l}\rho_k f_k\right\|_{L^2(\Omega)}, \text{ for any }l =1,\dots, N_k, \label{lowermass}\\
    \begin{gathered}\label{recursive}
    \frac{B_k\lambda_k^{1/p_c}}{2(\log\lambda_k)^{\veps_1}} \leq
    \alpha_k \left|\left\{\left|\tilde{\sigma}_k\left(\rho_{k}f_k-\sum_{l=1}^{L}Q_{\nu_l}\rho_k f_k\right) \right|>\frac{\alpha_k}{2} \right\}\right|^{1/p_c},
\\
\text{for any } L=1,\dots, N_k,
\end{gathered}
 \end{gather}
and in the latter case, $\tilde{\sigma}_k$ acts on the function in parentheses.  We now show how to derive a contradiction assuming these two hold.  Recall that the integral kernel of $\tilde{\sigma}_k$ is supported in $\Omega\times\Omega$.  Hence \eqref{recursive} and the classical $L^{p_c}$ bounds of the second author on $\tilde{\sigma}_k$ in \eqref{linearestimates} gives
\begin{equation*}
 \frac{B_k\lambda_k^{1/p_c}}{2(\log\lambda_k)^{\veps_1}} \lesssim \lambda_k^{1/p_c}\left\|\rho_{k}f_k-\sum_{l=1}^{N_k}Q_{\nu_l}\rho_kf_k\right\|_{L^{2}(\Omega)}.
\end{equation*}
We now multiply by $\lambda_k^{-1/p_c}$, square both sides, and apply \eqref{aocollection} and \eqref{lowermass} to obtain
\begin{align*}
 \left(\frac{B_k}{2(\log\lambda_k)^{\veps_1}}\right)^2 &\lesssim
 4\|\rho_kf_k\|_{L^2(\Omega)}^2 -\sum_{l=1}^{N_k} \|
 Q_{\nu_l}\rho_{k}f_k\|_{L^2(\Omega)}^2 \\
 &\leq 4 -N_k\left( \frac{B_k}{2C(\log\lambda_k)^{\veps_1}}\right)^{2/\delta_d}.
\end{align*}
Here we have used that $\|f_k\|_{L^2} =1$ and our assumption $\|\rho\|_{L^\infty} \leq 1$ (cf. \eqref{rhoLinfty}). Since $N_k$ is selected so that the right hand side is negative, we obtain a contradiction.

To see how to construct $Q_{\nu_1}\rho_kf_k, \dots, Q_{\nu_{N_k}}\rho_kf_k$, we proceed inductively.  For any $L=1, \dots, N_k-1$, we show how to select the successive  function in the collection given the previously chosen $Q_{\nu_1}\rho_kf_k, \dots, Q_{\nu_{L}}\rho_kf_k$ which satisfy \eqref{lowermass}, \eqref{recursive}.  The initial selection of $Q_{\nu_1}$ is essentially the same, simply take $h_2 =0$ in the following argument.  Denote
\begin{equation*}
 h_1 = \rho_{k}f_k-\sum_{l=1}^{L}Q_{\nu_l}\rho_k f_k = \sum_{\nu \neq \nu_l}Q_{\nu_l}\rho_k f_k, \qquad h_2 =\sum_{l=1}^{L}Q_{\nu_l}\rho_k f_k,
\end{equation*}
where the second expression for $h_1$ is a sum over all $\nu$ distinct from the $\nu_1,\dots,\nu_L$.  Then, by our assumptions
\begin{equation}\label{distsplit}
 \frac{B_k\lambda_k^{1/p_c}}{(\log\lambda_k)^{\veps_1}} <
 \alpha_k \left|\left\{ |\tilde{\sigma}_k h_1|>\frac{\alpha_k}{2}\right\} \right|^{\frac{1}{p_c}} +
 \alpha_k \left|\left\{ |\tilde{\sigma}_k h_2|>\frac{\alpha_k}{2}\right\} \right|^{\frac{1}{p_c}}.
\end{equation}

Our first main claim is that we can use Corollary \ref{C:microlp} to see that
\begin{equation}\label{smallnorms}
  \alpha_k \left|\left\{ |\tilde{\sigma}_k h_2|>\frac{\alpha_k}{2}\right\} \right|^{1/p_c} = o\left(\lambda_k^{1/p_c}(\log \lambda_k)^{-\veps_1}\right).
\end{equation}
We initially observe the following $L^2 \to L^{p_c}$ ``commutator bounds''
  \begin{equation}\label{sigmaQcommute}
    \left\|\tilde\sigma_\lambda Q_\nu - Q_\nu \tilde\sigma_\lambda \right\|_{L^2(\Omega)\to L^{p_c}(\Omega)} \lesssim \lambda^{-\frac 14}.
  \end{equation}
Morally, this is Sobolev embedding and Egorov's theorem combined with the invariance of $q_\nu$ under the flow $\chi_t$.  However, we give a direct proof below that will be shown after the related Lemma \ref{L:pdofiocomp}.  Assuming \eqref{sigmaQcommute}, we use properties of the distribution function and Chebyshev's inequality to get
\begin{equation*}
   \alpha_k^{p_c} \left|\left\{ |\tilde{\sigma}_k h_2|>\frac{\alpha_k}{2}\right\} \right| \lesssim \sum_{l=1}^{L}\alpha_k^{p_c}
   \left| \left\{|Q_{\nu_l}\tilde\sigma_k\rho_kf_k|>\frac{\alpha_k}{4L} \right\} \right| + L\lambda^{-\frac{p_c}{4}}.
\end{equation*}
Since the $1/p_c$ power of the second term on the right is much stronger than the bounds in \eqref{smallnorms}, we are left with estimating the first term on the right hand side.

We next observe that for any $p \in[1,\infty]$, $\|Q_{\nu_l}\|_{L^{p}\to L^{p}} \lesssim 1$, which will allow us to eliminate $\tilde{\sigma}_k$ and apply Corollary \ref{C:microlp}.  Indeed, given that $q_\nu$ satisfies \eqref{microcutoff} with $\omega(x)= \omega(x,\nu)$, integration by parts in the expression for the integral kernel $Q_{\nu_l}(x,y)$ yields the pointwise bounds
\begin{equation*}
  |Q_{\nu_l}(x,y)| \lesssim \lambda^{1+\frac{7}{8}(d-1)}(1+\lambda|\omega(x,\nu)\cdot(x-y)| + \lambda^{\frac 78}|x-y|)^{-(d+1)}.
\end{equation*}
Hence the uniform bounds on $L^p$ follow from the generalized Young's inequality.  By Chebyshev's inequality, \eqref{sigmaQcommute}, and \eqref{sigmareplace}, we now have
\begin{equation*}
   \alpha_k^{p_c}\left| \left\{|Q_{\nu_l}\tilde\sigma_k\rho_kf_k|>\frac{\alpha}{4L} \right\} \right| \lesssim
   \alpha_k^{p_c}\left| \left\{|Q_{\nu_l}\rho_kf_k|>\frac{\alpha}{8L} \right\} \right| + \lambda(\log\lambda)^{-p_c},
\end{equation*}
and as before the last term on the right is of the desired size in \eqref{smallnorms}.

We may now use that Corollary \ref{C:microlp} yields the following bound,
\begin{equation}\label{microlpconsequence}
 \sum_{l=1}^{L}\alpha_k^{p_c} \left| \left\{|Q_{\nu_l}\rho_kf_k|>\frac{\alpha}{8L} \right\} \right| \lesssim L^{1+p_c}\lambda_kc(\lambda_k)^{\frac{p_c}{d+1}} \leq \lambda_kN_k^{1+p_c}c(\lambda_k)^{\frac{p_c}{d+1}}.
\end{equation}
To see that \eqref{smallnorms} now follows, take $p_c$-th roots of both sides here and recall that $N_k = o((\log\lambda_k)^{2\veps_1/\delta_d})$. The condition on the exponent in Remark \ref{R:optimalexponentintro} and the relation $\veps_1 = \frac{(d+1)\veps_0}{2}$ then implies
\begin{equation}\label{exponentlimitation}
  \begin{cases}
    \frac{2\veps_1}{\delta_2}(\frac{p_c+1}{p_c}) - \frac{1}{2\cdot 3} = -\veps_1, & \mbox{if } d=2, \\
    \frac{2\veps_1}{\delta_3}(\frac{p_c+1}{p_c}) - \frac{1}{4} < -\veps_1, & \mbox{if } d= 3,\\
    \frac{2\veps_1}{\delta_d}(\frac{p_c+1}{p_c}) - \frac{1}{d+1} = -\veps_1, & \mbox{if } d\geq 4.
  \end{cases}
\end{equation}
We show the details behind this when $d \geq 4$ so that $c(\lambda) = (\log\lambda)^{-1}$, and note that the other cases are verified similarly.  Given the prior observation on the size of $N_k$, the $p_c$-th root of the right hand side of \eqref{microlpconsequence} is
\begin{equation*}
\mathcal{O}\left(\lambda_k^{\frac{1}{p_c}}N_k^{1+\frac{1}{p_c}}c(\lambda_k)^{\frac{1}{d+1}}\right) = o\left(\lambda_k^{\frac{1}{p_c}}(\log\lambda_k)^{\frac{2\veps_1}{\delta_d}(\frac{p_c+1}{p_c})}
(\log\lambda_k)^{-\frac{1}{d+1}}\right).
\end{equation*}
It is now an easy algebraic computation to see that the choice of $\veps_0$ in Remark \ref{R:optimalexponentintro} means that $\veps_1$ satisfies \eqref{exponentlimitation}.  The improvements on the exponent for negatively curved manifolds claimed in Remarks \ref{R:optimalexponentintro} and \ref{R:negcurvature} follow since the equation for $d \geq 4$ in \eqref{exponentlimitation} is now satisfied for $d=2,3$.


Given \eqref{distsplit} and \eqref{smallnorms}, for $k$ large enough and independently of $L$,
\begin{equation}\label{nextchoiceprep}
 \frac 34\frac{B_k\lambda_k^{1/p_c}}{(\log\lambda_k)^{\veps_1}} < \alpha_k \left|\left\{ \left|\tilde{\sigma}_\lambda h_1\right|>\frac{\alpha_k}{2}\right\} \right|^{\frac{1}{p_c}}.
\end{equation}
We are now left to show that there exists $Q_{\nu_{L+1}}$, distinct from those previously chosen, such that $Q_{\nu_{L+1}}\rho_k f_k$ also satisfies the bounds in \eqref{lowermass}, i.e.,
\begin{equation}\label{nextchoice}
  \frac 12\frac{B_k}{(\log\lambda_k)^{\veps_1}} \leq C\|Q_{\nu_{L+1}}\rho_k f_k \|_{L^2(\Omega)}^{\delta_d}.
\end{equation}
Indeed, once this is shown \eqref{recursive} can be concluded by taking $h_1 = \rho_{k}f_k-\sum_{l=1}^{L+1}Q_{\nu_l}\rho_k f_k$ and $h_2=\sum_{l=1}^{L+1}Q_{\nu_l}\rho_k f_k$ in \eqref{distsplit} and using \eqref{smallnorms} once again.

Given \eqref{L2bounds}, we have $\|h_1\|_{L^2} \leq 4$ and hence by \eqref{nextchoiceprep} and Theorem \ref{T:bubblethm}, there exists $\nu_{\max}$ such that
\begin{equation}\label{nextchoiceTheoremapp}
   \frac 34\frac{B_k\lambda_k^{1/p_c}}{(\log\lambda_k)^{\veps_1}} < C(2C_d)^{-\delta_d}\lambda^{1/p_c}\left\| \sum_{\nu\neq\nu_l} Q_{\nu_{\max}}Q_{\nu}\rho_k f_k\right\|_{L^2(\Omega)}^{\delta_d},
\end{equation}
where the sum in the expression on the right is over all $\nu$ distinct from each of the $\nu_l$, $l =1, \dots,L$.  Here we have used our assumption that $C(2C_d)^{-\delta_d}$ exceeds the implicit constant in \eqref{bubblebd}.  Now take $\nu_{L+1}$ so that
\begin{equation}\label{maxbubble}
  \| Q_{\nu_{\max}}Q_{\nu_{L+1}}\rho_k f_k\|_{L^2} = \max\left\{\| Q_{\nu_{\max}}Q_{\nu}\rho_k f_k\|_{L^2}: \nu \neq \nu_1,\dots, \nu_L\right\},
\end{equation}
so that $Q_{\nu_{L+1}}$ is distinct from the previously chosen operators.  Note that by the symbolic calculus
\begin{equation}\label{disjointsupp}
  \| Q_{\nu}\circ Q_{\nu'}\|_{L^2 \to L^2} \lesssim_N \lambda^{-N}, \quad \text{ if }\supp(q_{\nu})\cap \supp(q_{\nu'})= \emptyset.
\end{equation}
We therefore must have $\supp(q_{\nu_{\max}})\cap \supp(q_{\nu_{L+1}}) \neq \emptyset$ in \eqref{maxbubble}, since otherwise \eqref{disjointsupp} would imply a contradiction of \eqref{nextchoiceTheoremapp}.  Hence \eqref{cdconstant}, \eqref{maxbubble}, and taking $N > 1+\frac{d-1}{8}$ in \eqref{disjointsupp} then yields for $k$ large enough
\begin{align*}
  \left\| \sum_{\nu\neq\nu_l} Q_{\nu_{\max}}Q_{\nu}\rho_k f_k\right\|_{L^2(\Omega)}
  &\leq C_d \| Q_{\nu_{\max}}Q_{\nu_{L+1}}\rho_k f_k\|_{L^2} + \lambda_k^{-1}\\
  &\leq 2C_d \| Q_{\nu_{L+1}}\rho_k f_k\|_{L^2}+\lambda_k^{-1}.
\end{align*}
Combining this with \eqref{nextchoiceTheoremapp} then gives \eqref{nextchoice} for large enough $k$.


\begin{remark}\label{R:optimalexponent}
  The condition \eqref{exponentlimitation} is the strongest limitation on $\veps_1$, which in turn gives $\veps_0$ as in \eqref{eps0}.  Indeed, the only other assumption was $\veps_1\leq \frac{1}{d+1}<1$ in \eqref{initwkblowup} and \eqref{sigmareplace}.  Also, with some small changes in the exposition, the arguments here show that in the $d=3$ case of Theorem \ref{T:mainthm}, we have
  \begin{equation*}
          \| \rho(T(\lambda - P))\|_{L^2(M) \to L^{4}(M)} \lesssim \frac{\lambda^{1/p_c}(\log\log\lambda)^{1/4}}{(\log\lambda)^{1/48}}.
  \end{equation*}
\end{remark}
\section{Almost orthogonality and bilinear estimates}\label{S:aobilinear}
Here we prove Theorem \ref{T:bubblethm}, which involves bilinear estimates and almost orthogonality in the spirit of the prior works of the authors \cite{SoggeKaknik}, \cite{BlairSoggeKaknik}, \cite{BlairSoggeRefined}, \cite{BlairSoggeRefinedHighDimension}.

\subsection{Whitney-type decompositions and the key lemmas} Recall that $\nu$ indexes a $\approx \lambda^{-1/8}$ separated set in a neighborhood of $(1,0,\dots,0)$ on $\mathbb{S}^{d-1}$. Given \eqref{sigmadecomp}, we may write
\begin{equation}\label{sigmasquared}
  (\tilde{\sigma}_\lambda h)^2  = \sum_{\nu,\tilde \nu}
  (\tilde{\sigma}_\lambda Q_\nu h)(\tilde{\sigma}_\lambda Q_{\tilde \nu} h).
\end{equation}
We may thus view this neighborhood of $(1,0,\dots,0)$  as a graph in the last $d-1$ variables, and given $\nu, \tilde{\nu} \in \mathbb{S}^{d-1}$, we let $\nu', \tilde{\nu}'$ denote the projection of these vectors onto the last $d-1$ coordinates.  This allows us to organize the sum here in a fashion similar to that in \cite[p.513]{BlairSoggeRefinedHighDimension}, which in turn is analogous to the Whitney decomposition taken in \cite[\S2.5]{TaoVargasVega}.

Consider the standard family of dyadic cubes in $\RR^{d-1}$ with $\tau_{\mu'}^j$ denoting the translation of $[0,2^j)^{d-1}$ by $\mu'\in 2^j\mathbb{Z}^{d-1}$.  Two dyadic cubes of sidelength $2^j$ are declared to be \emph{close} if they are not adjacent, but have adjacent parents of sidelength $2^{j+1}$, and in this case we write $\tau_{\mu'}^j \sim \tau_{\tilde \mu'}^j$.  Note that close cubes satisfy $d(\tau_{\mu'}^j , \tau_{\tilde \mu'}^j) \approx 2^j$.  As noted in \cite[p.971]{TaoVargasVega}, any two distinct points $\nu',\tilde\nu' \in \RR^{d-1}$ lie in a unique pair of close cubes, that is, there exists a unique triple $j,\mu,\mu'$ such that $(\nu',\tilde\nu') \in \tau_{\mu'}^j \times \tau_{\tilde \mu'}^j$ and $\tau_{\mu'}^j \sim \tau_{\tilde \mu'}^j$.  We remark that in what follows we only need to consider $j \leq 0$. 


Let $J$ be the integer satisfying $2^{J-1}<8\lambda^{-1/8} \leq 2^J$. The observations above now imply that the sum in \eqref{sigmasquared} can be organized as
\begin{equation}\label{sumorganized}
\left(  \sum_{j=J+1}^{0} \sum_{(\nu',\tilde \nu') \in \tau_{\mu'}^j \times \tau_{\tilde \mu'}^j: \tau_{\mu'}^j \sim \tau_{\tilde \mu'}^j}+ \sum_{(\nu',\tilde \nu')\in \Xi_J}\right)
  (\tilde{\sigma}_\lambda Q_\nu h)(\tilde{\sigma}_\lambda Q_{\tilde \nu} h)
\end{equation}
where $\Xi_J$ indexes the remaining pairs such that $|\nu'-\tilde\nu'| \lesssim \lambda^{-1/8}$, including the on-diagonal pairs $\nu'= \tilde\nu'$.  To see that this does indeed rewrite the sum in \eqref{sigmasquared}, note that if $\nu'\neq\tilde\nu'$, then as observed above $(\nu',\tilde\nu') \in \tau_{\mu'}^j \times \tau_{\tilde \mu'}^j$ for some unique pair of close cubes.  If $j \leq J$, then we say $(\nu',\tilde\nu') \in \Xi_J$.  Otherwise, it is included in the first sum in \eqref{sumorganized}.  Note that
\begin{equation}\label{XiJcardinal}
\text{ for each $\nu'$, }\; \#\{\tilde{\nu}':(\nu',\tilde\nu') \in \Xi_J \} = \mathcal{O}(1).
\end{equation}
For $j >J$, we define $\Xi_j$ differently, indexing
\begin{equation*}
\Xi_j:=  \{(\mu', \tilde \mu') \in 2^j\mathbb{Z}^{2(d-1)}: \tau_{\mu'}^j \sim \tau_{\tilde \mu'}^j \}.
\end{equation*}
Let $\mu\in \mathbb{S}^{d-1}$ be the vector with positive first coordinate and last $d-1$ coordinates given by $\mu'$. Define
\begin{equation*}
 Q_{j,\mu}h := \sum_{\nu' \in \tau^j_{\mu'}} Q_\nu h,
\end{equation*}
so that
\begin{equation*}
  \sum_{(\nu',\tilde \nu') \in \tau_{\mu'}^j \times \tau_{\tilde \mu'}^j: \tau_{\mu'}^j \sim \tau_{\tilde \mu'}^j}(\tilde{\sigma}_\lambda Q_\nu h)(\tilde{\sigma}_\lambda Q_{\tilde \nu} h) =
  \sum_{(\mu',\tilde \mu')\in \Xi_j}  (\tilde{\sigma}_\lambda Q_{j,\mu}h)(\tilde{\sigma}_\lambda Q_{j,\tilde \mu}h).
\end{equation*}

Now define a semiclassical symbol $\tilde{q}_{j,\mu}$ satisfying
\begin{equation}
\begin{gathered}
\tilde{q}_{j,\mu}(x,\xi) \Bigg( \sum_{\nu' \in \tau^j_{\mu'}} \tilde{q}_\nu (x,\xi) \Bigg)= \sum_{\nu' \in \tau^j_{\mu'}} \tilde{q}_\nu (x,\xi) ,\\
  \left|\langle \omega(x, \mu), d_\xi\rangle^j \prtl^\beta_{x,\xi}  \tilde{q}_{j,\nu}\right|\lesssim_{\beta, j} 2^{-j|\beta|},\\
  \supp(\tilde{q}_{j,\mu}) \subset \left\{(x,\xi) \in T^* \Omega: \left|\xi/|\xi|_{g(x)}-\omega(x,\mu) \right|_{g(x)} \lesssim 2^j, |\xi| \approx 1\right\},\label{suppqtilde}
  \end{gathered}
\end{equation}
where $\omega(x,\mu)\in S_x^*\Omega$ is the covector of the unit speed geodesic passing through $x$ and whose covector takes the form $\mu/|\mu|_{g(x)}$ as it passes through the $x_1=0$ plane (see the discussion following \eqref{planeinverse}).  As usual, denote $\tilde{Q}_{j,\mu}:=\text{Op}(\tilde{q}_{j,\mu})$ with the usual quantization \eqref{standardquant}. Taking the support of $\tilde{q}_{j,\mu}$ suitably we may assume
\begin{equation}\label{supportprops}
  \begin{gathered}
    d\big(\supp(\tilde{q}_{j,\mu}),\supp(\tilde{q}_{j,\tilde \mu})\big) \approx 2^j, \quad \text{for } (\mu',\tilde \mu') \in \Xi_j, \;j=J+1,\dots,0,\\
    \|(I-\tilde{Q}_{j,\mu})\circ Q_{j,\mu} h\|_{L^2} \lesssim_N \lambda^{-N}.
  \end{gathered}
\end{equation}

Next define the bilinear operators
\begin{equation*}
\begin{gathered}
    \Upsilon(h_1,h_2)(x) :=
    \big(\tilde{\sigma}_\lambda h_1\big)(x)\big(\tilde{\sigma}_\lambda h_2\big)(x), \\
    \Upsilon_{j,\mu,\tilde \mu}(h_1,h_2) := \Upsilon(\tilde{Q}_{j,\mu}h_1,\tilde{Q}_{j,\tilde \mu} h_2),
    \qquad j=J+1,\dots,0.
\end{gathered}
\end{equation*}
This allows \eqref{sigmasquared} to be rewritten as
\begin{gather}
(\tilde{\sigma}_\lambda h)^2 =  \Upsilon^{\diag}(h) + \Upsilon^{\text{off}}(h)+ \Upsilon^{\text{smooth}}(h),\notag\\
    \Upsilon^{\diag}(h):= \sum_{(\nu',\tilde \nu')\in \Xi_J}(\tilde{\sigma}_\lambda Q_\nu h)(\tilde{\sigma}_\lambda Q_{\tilde \nu} h),\label{tripledecomp}\\
    \Upsilon^{\text{off}}(h):= \sum_{j=J+1}^{0} \sum_{(\mu,\tilde{\mu}) \in \Xi_j} \Upsilon_{j,\mu,\tilde{\mu}}\left({Q}_{j,\mu}h,{Q}_{j,\tilde \mu}h\right),\notag\\
    \Upsilon^{\text{smooth}}(h):= \sum_{j=J+1}^{0} \sum_{(\mu,\tilde{\mu}) \in \Xi_j} \left(\Upsilon\left(Q_{j,\mu}h,Q_{j,\tilde\mu}h\right)-
    \Upsilon_{j,\mu,\tilde{\mu}}\left(Q_{j,\mu}h, Q_{j,\tilde\mu}h\right)\right).\notag
\end{gather}
Each term in the sum defining $\Upsilon^{\text{smooth}}$ can be rewritten as a sum of 3 terms, each of which contains a factor of the form $\sigma_\lambda((I-\tilde{Q}_{j,\mu})\circ Q_{j,\mu} h)$ (or one with $\mu$ replacing $\tilde \mu$).  Hence linear estimates on $\sigma_\lambda$ in \eqref{linearestimates}, almost orthogonality in \eqref{crudebound}, and taking $N$ large in \eqref{supportprops} implies
\begin{equation*}
  \left\|  \Upsilon^{\text{smooth}}(h)\right\|_{L^{p_c/2}} \lesssim \lambda^{-N}\|h\|_{L^2}^2 \lesssim \lambda^{\frac{2}{p_c}}\left( \max_{\nu}\|Q_\nu h\|_{L^2}^{\delta_d} \right)^2.
\end{equation*}
Hence by \eqref{tripledecomp}
\begin{multline*}
  \left| \left\{ |\tilde{\sigma}_\lambda h| > \alpha \right\}\right| \leq
  \left| \left\{ |\Upsilon^{\diag}(h)| > \frac{\alpha^2}{3} \right\}\right|\\ +
  \left| \left\{ |\Upsilon^{\text{off}}(h)| > \frac{\alpha^2}{3} \right\}\right| + \left| \left\{ |\Upsilon^{\text{smooth}}(h)| > \frac{\alpha^2}{3} \right\}\right|.
\end{multline*}
As observed, the last term here is easily bounded by Chebyshev's inequality.  The following lemma shows that $\Upsilon^{\text{off}}$ satisfies stronger estimates as well and is closely related to \cite[Theorem 2.1]{BlairSoggeKaknik} and \cite[Theorems 3.3, 3.4]{BlairSoggeRefinedHighDimension}:
\begin{lemma}\label{L:offdiag} Suppose $\frac{2(d+2)}{d} < q < \frac{2(d+1)}{d-1}=p_c$.  Then for $j= J+1,\dots, 0$,
\begin{equation}\label{offdiag}
  \left\|\sum_{(\mu,\tilde{\mu}) \in \Xi_j} \Upsilon_{j,\mu,\tilde{\mu}}\left(Q_{j,\mu}h,Q_{j,\tilde{\mu}}h\right)\right\|_{L^{q/2}}
  \lesssim_q \lambda^{d-1 - \frac{2d}{q}}2^{j(d-1-\frac{2(d+1)}{q})}
  \left\| h\right\|_{L^2}^2.
\end{equation}
\end{lemma}
To appreciate the gain furnished by this lemma in weak-$L^{p_c}$, the first two factors on the right in \eqref{offdiag} should be raised to the power of $\frac{q}{2p_c}$, which is $\lambda^{\frac{1}{p_c}}(\lambda 2^j)^{\frac{d-1}{2}(\frac{q}{p_c}-1)}$.
Summation in $j$  thus gives for any fixed $q \in (\frac{2(d+2)}{d},p_c)$
\begin{align}
  \alpha\left| \left\{ |\Upsilon^{\text{off}}(h)| > \frac{\alpha^2}{3} \right\}\right|^{\frac{1}{p_c}} &\lesssim \alpha^{1-\frac{q}{p_c}}\lambda^{\frac{1}{p_c}}
  (\lambda^{\frac 78})^{\frac{d-1}{2}(\frac{q}{p_c}-1)}\left\| h\right\|_{L^2}^{\frac{q}{p_c}}\label{crucialoffbd}\\
  &\lesssim \lambda^{\frac{1}{p_c}}\big(\alpha \lambda^{\frac{7-7d}{16}}\big)^{1-\frac{q}{p_c}}. \notag
\end{align}
Since $\alpha \leq \lambda^{\frac{d-1}{4}+\frac 18}$, the quantity in parentheses in the last line is $\mathcal{O}(\lambda^{\frac{5-3d}{16}})$.  Thus the right hand side can be bounded by the second term on right hand side of \eqref{bubblebd}.

A step in the proof of \eqref{offdiag} and in the treatment of $\Upsilon^{\diag}(h)$ is to show the following almost orthgonality lemma, akin to \cite[Theorem 3.3]{BlairSoggeRefinedHighDimension}.  This establishes an almost orthogonality principle in $L^r$ spaces, but with respect to the operators $\Upsilon_{j,\mu,\mu'}$ and their counterparts in the definition of $\Upsilon^\diag$. Hence this is a substantial variation on the $L^2$ almost orthogonality principle in Proposition \ref{prop:aoL2}.
\begin{lemma}\label{L:aobound}
For $1 \leq r \leq \infty$, set $r^* = \min(r,r')$ where $r'$ is the H\"older conjugate of $r$.  Then for any $j=J+1,\dots, 0$, and any $N \in \mathbb{N}$ large
\begin{multline}\label{aobound}
    \left\|\sum_{(\mu,\tilde{\mu}) \in \Xi_j} \Upsilon_{j,\mu,\tilde{\mu}}\left(Q_{j,\mu}h,Q_{j,\tilde{\mu}}h\right)\right\|_{L^{r}}\\
    \lesssim_N
    \left( \sum_{(\mu,\tilde{\mu}) \in \Xi_j} \left\| \Upsilon_{j,\mu,\tilde{\mu}}\left(Q_{j,\mu}h,Q_{j,\tilde{\mu}}h\right)\right\|_{L^{r}}^{r^*}\right)^{1/r^*} + \lambda^{-N}\|h\|_{L^{r}},
\end{multline}
\begin{equation}\label{aobounddiag}
  \left\|\Upsilon^{\diag}(h)\right\|_{L^{r}} \lesssim_N
  \left( \sum_{(\nu,\tilde{\nu}) \in \Xi_J}
  \left\| (\tilde{\sigma}_\lambda Q_\nu h)(\tilde{\sigma}_\lambda Q_{\tilde \nu} h)\right\|_{L^{r}}^{r^*}\right)^{1/r^*}+ \lambda^{-N}\|h\|_{L^{r}}.
\end{equation}
\end{lemma}
Given Lemma \ref{L:aobound}, we have that with $r=p_c/2$,
\begin{equation*}
\alpha\left| \left\{ |\Upsilon^{\diag}h| > \frac{\alpha^2}{3} \right\}\right|^{\frac{1}{p_c}}
\lesssim\left( \sum_{(\nu,\tilde{\nu}) \in \Xi_J}
  \left\| (\tilde{\sigma}_\lambda Q_\nu h)(\tilde{\sigma}_\lambda Q_{\tilde \nu} h)\right\|_{L^{r}}^{r^*}\right)^{1/r^*}+\lambda^{-N}\|h\|_{L^{r}}.
\end{equation*}
The factor $\lambda^{-N}$ means that the second term here can be harmlessly absorbed in to the second term in \eqref{bubblebd}.  Indeed, when $d \geq 3$, $r=p_c/2\leq 2$ so H\"older's inequality can be applied as $h$ is supported in $\Omega$.  When $d =2$, Sobolev embedding can be applied instead.  For the first sum on the right, the linear estimates \eqref{linearestimates} for $\tilde{\sigma}_\lambda $ shows it is bounded by
\begin{equation}\label{linearapp}
\lambda^{\frac{1}{p_c}} \left( \sum_{(\nu,\tilde{\nu}) \in \Xi_J}
\|Q_\nu h\|_{L^{2}}^{r^*}
\|Q_{\tilde{\nu}} h\|_{L^{2}}^{r^*}\right)^{\frac{1}{2r^*}}.
\end{equation}
This sum can be treated very similarly to \cite[\S3]{BlairSoggeRefinedHighDimension} and previous works.  We show the details for $d \geq 4$, as the other cases are handled similarly.  When $d \geq 4$, we have $p_c/2<2$ and hence $(p_c/2)^* = p_c/2$.  Using H\"older's inequality with $\frac{1}{r^*}=\frac{2}{p_c} = \frac 12 + \frac{4-p_c}{2p_c}$ and \eqref{XiJcardinal}, the sum is bounded by
\begin{multline*}
  \left(\sum_\nu \|Q_\nu h\|_{L^{2}}^2\right)^{\frac 14}
  \left(\sum_{\tilde{\nu}} \|Q_{\tilde{\nu}} h\|_{L^{2}}^{\frac{2p_c}{4-p_c}}\right)^{\frac{4-p_c}{4p_c}} \\
  \lesssim \left(\sum_\nu \|Q_\nu h\|_{L^{2}}^2\right)^{\frac{1}{p_c}}
  \left( \sup_\nu \|Q_\nu h \|_{L^2}^{1-\frac{2}{p_c}} \right) \lesssim
  \|h\|_{L^{2}}^{\frac{2}{p_c}}\left( \sup_\nu \|Q_\nu h \|_{L^2}^{1-\frac{2}{p_c}} \right) ,
\end{multline*}
and in the first inequality we use $\frac{2p_c}{4-p_c} = 2 + \frac{4p_c-8}{4-p_c}$ and \eqref{aoproperty}. When $d=3$, $\frac{p_c}{2}=2$ so we use H\"older's inequality again but with $\frac{2}{p_c} = \frac 12 + \frac{1}{\infty}$, making the first inequality extraneous. When $d=2$, the argument is similar to the $d\geq 4$ case, but $(p_c/2)^* = 3/2< p_c/2$, and the resulting exponent is $\delta_0 = \frac 13$.  This completes the proof of Theorem \ref{T:bubblethm} once we show Lemmas \ref{L:offdiag} and \ref{L:aobound}.

\subsection{Composition of $\tilde{\sigma}_\lambda$ with PDO and bilinear estimates}\label{SS:pdobilinear}  Here we prove Lemma \ref{L:offdiag}, assuming Lemma \ref{L:aobound} for now. As stated above, it is nearly the same as \cite[Theorem 2.1]{BlairSoggeKaknik} or \cite[Theorems 3.3, 3.4]{BlairSoggeRefinedHighDimension}, but we are somewhat thorough here as there are differences in the constructions.

Given Lemma \ref{L:aobound}, we are reduced to showing that for $\frac{2(d+2)}{d} < q < \frac{2(d+1)}{d-1}$,
\begin{equation}\label{bilinearaoreduction}
  \left\| \Upsilon_{j,\mu,\tilde{\mu}}\left(Q_{j,\mu}h,Q_{j,\tilde{\mu}}h\right)\right\|_{L^{q/2}}
  \lesssim_q \lambda^{d-1 - \frac{2d}{q}}2^{j(d-1-\frac{2(d+1)}{q})}
  \left\| Q_{j,\mu}h\right\|_{L^2}\left\| Q_{j,\tilde \mu}h\right\|_{L^2}.
\end{equation}
Indeed, if this holds, then given that for each $\mu \in 2^j\mathbb{Z}^{d-1}$ there are $\mathcal{O}(1)$ elements $\tilde \mu$ satisfying $(\mu,\tilde{\mu}) \in \Xi_j$ (similar to \eqref{XiJcardinal}), hence Lemma \ref{L:aobound} with $r=q/2$  and Cauchy-Schwarz means it suffices to bound
\begin{equation*}
\left(\sum_{(\mu,\tilde{\mu}) \in \Xi_j} \left\| Q_{j,\mu}h\right\|_{L^2}^{(\frac q2)^*}\left\| Q_{j,\tilde \mu}h\right\|_{L^2}^{(\frac q2)^*}\right)^{\frac{1}{(\frac q2)^*}} \lesssim \left(\sum_{\mu \in 2^j\mathbb{Z}^{d-1}} \left\| Q_{j,\mu}h\right\|_{L^2}^{2(\frac q2)^*}\right)^{\frac{1}{(\frac q2)^*}}.
\end{equation*}
But given the almost orthogonality of the $\{Q_{j,\mu}\}_{\mu}$ (proved similarly to \eqref{L2bounds} and \eqref{aoproperty}) and the embedding $\ell^2 \hookrightarrow \ell^{2(\frac q2)^*}$, the right hand side is $\mathcal{O}(\|h\|_{L^2}^2)$.

\subsubsection{Composition of $\tilde{\sigma}_\lambda$ with PDO} To set the stage for bilinear estimates, we need to examine the effect of composing $\tilde{\sigma}_\lambda$ with the $\tilde{Q}_{j,\mu}$. Recall that $\omega(z,\mu)$ is the covector at $z$ of the geodesic through $z$ whose covector at the intersection of the $y^1=0$ plane is $\mu/|\mu|_g$.  We use ``$\#$'' to denote the isomorphism from $T_x^*M$ to $T_x M$ determined by the metric $g$ (the ``musical isomorphism'').  Its inverse is denoted by ``$\flat$''.

\begin{lemma}\label{L:pdofiocomp}
For any $N \in \mathbb{N}$, the kernel of $\tilde{\sigma}_\lambda\circ \tilde{Q}_{j,\mu}$ can be written
\begin{equation}\label{pdofiocomp}
  (\tilde{\sigma}_\lambda\circ \tilde{Q}_{j,\mu})(x,z) =
  \lambda^{\frac{d-1}{2}}e^{i\lambda\Theta_g(x,z)}V_{j,\mu}(x,z) + \mathcal{O}(\lambda^{-N}),
\end{equation}
where $V_{j,\mu}=0$ unless $\Theta_g(x,z) \in (\frac{\tilde{c}_0}{4},\frac{\tilde{c}_0}{2})$ and the unit covector at $z$ of the geodesic from $z$ to $x$ lies in $\supp(\tilde{q}_{j,\mu}(z,\cdot))$.  Moreover, denoting $(\omega(x,\mu)^\#)^k$ as $k$ applications of the vector field obtained by raising the indices of $\omega(x,\mu)$,
\begin{equation}\label{pdofiocompest}
  \left|(\omega(x,\mu)^\#)^k\prtl^\beta_{x}V_{j,\mu}(x,z)\right| \lesssim_{k,\beta} 2^{-j|\beta|} \qquad \text{for any } k, \beta.
\end{equation}
\end{lemma}
\begin{proof}  This is a small variation on the stationary phase arguments in \cite[Lemma 5.1.3]{soggefica}.  The kernel $ (\tilde{\sigma}_\lambda\circ \tilde{Q}_{j,\mu})(x,z)$ is given by
\begin{equation}\label{pdofiocompkernel}
 \frac{\lambda^{2d}}{(2\pi)^d}\int e^{i\lambda(t+\varphi(t,x,\xi) -y\cdot \xi + (y-z)\cdot \eta)}\hat{\rho}(t/\tilde c_0)v(t,x,\xi) \tilde{\psi}(y)\tilde{q}_{j,\mu}(y,\eta)\,dtdyd\eta d\xi.
\end{equation}
The critical points of this oscillatory integral satisfy
\begin{equation}\label{critptspdofio}
 \eta = \xi, \;y=z, \;p(x,d_x \varphi(t,x,\xi)) = 1, \;y = d_\xi \varphi(t,x,\xi),
\end{equation}
which arises from differentiation in $y,\eta, t,\xi$ respectively. The third identity here uses the eikonal equation for $\varphi$.  The last 2 identities fix $t$, $y$, $\xi$ so that $\xi$ lies in the cosphere bundle and $t=\Theta_g(x,z)$ is the time at which the minimizing geodesic through $(y,\xi)=(z,\xi)$ passes through $(x,d_x\varphi(t,x,\xi))$.  Hence the kernel is $\mathcal{O}(\lambda^{-N})$ when $\Theta_g(x,z) \notin (\tilde c_0/4,\tilde c_0/2)$.  Moreover, at the critical points, $\varphi(t,x,\xi) -y\cdot \xi=0$ since $\varphi$ is homogeneous of degree 1 in $\xi$.  Stationary phase can be applied to \eqref{pdofiocompkernel} since the mixed Hessian in $(t,\xi)$ is nonsingular, which follows from the same idea as in \cite[p.140]{soggefica}.  This yields the expression \eqref{pdofiocomp} and the claim concerning the support of $V_{j,\mu}$.

We are left to verify that applying $\omega(x,\mu)^\#$ to $V_{j,\mu}$ yields no loss in $2^{-j}$.  Note that $\xi = \eta$ as a function of $(x,z)$ is the unit covector over $z$ of the geodesic joining $x$ and $z$.  If $\gamma(t)$ parameterizes this geodesic with $\gamma(0)=z$, $\xi(\gamma(t),z)$ is constant in $t$ and using the summation convention we have
\begin{equation*}
  0=\prtl_t\big(\xi_i(\gamma(t),z)\big) = \prtl_{x^k}\xi_i(\gamma(t),z)\dot{\gamma}^k(t)= \prtl_{x^k}\xi_i(\gamma(t),z)g^{jk}(\gamma(t))(\dot{\gamma}^\flat)_j(t),
\end{equation*}
where $\dot{\gamma}^\flat$ gives the unit covector of the geodesic at $t$.  Since $V_{j,\mu}=0$ unless $|\dot{\gamma}^\flat(0)-\omega(z,\mu)| \lesssim 2^j$, this shows the rest of \eqref{pdofiocompest}.
\end{proof}
\begin{proof}[Proof of \eqref{sigmaQcommute}]
The main idea is that the leading order term in the stationary phase expansion of the integral kernels of $\tilde{\sigma}_\lambda\circ Q_{\nu}$ and $Q_{\nu}\circ \tilde{\sigma}_\lambda$.  The remainder terms are then $\mathcal{O}(\lambda^{-3/4})$, at which point Sobolev embedding yields the desired gain for the remainder as $\frac{d}{2}-\frac{d}{p_c}= \frac{1}{p_c}+\frac 12$.

The kernel of $\tilde{\sigma}_\lambda\circ Q_{\nu}$ is just \eqref{pdofiocompkernel} but with $\tilde{q}_{j,\mu}(y,\eta)$ replaced by $q_\nu(y,\eta)$.  The critical points are thus determined by \eqref{critptspdofio} and since we have $\eta = \xi$ and $y=z$, we only need to consider $\xi$ and $t$ as functions of $x,z$.  Recall that $t(x,z) = \Theta_g(x,z)$.  Furthermore, the equation $z = d_\xi \varphi(t,x,\xi)$ means that $\xi(x,z)$ is the covector at $z$ of the unit speed geodesic from $z$ to $x$.

On the other hand, the kernel of $Q_{\nu}\circ \tilde{\sigma}_\lambda$ is
\begin{equation*}
 \frac{\lambda^{2d}}{(2\pi)^d}\int e^{i\lambda(t+\varphi(t,y,\xi) -z\cdot \xi + (x-y)\cdot \eta)}\hat{\rho}(t/\tilde c_0)v(t,y,\xi) \tilde{\psi}(z)q_{\nu}(x,\eta)\,dtdyd\eta d\xi.
\end{equation*}
This time the critical points of the phase are given by
\begin{equation*}
 \eta = d_y\varphi(t,y,\xi), \;x=y, \;p(y,d_y \varphi(t,y,\xi)) = 1, \;z = d_\xi \varphi(t,y,\xi).
\end{equation*}
This time we must treat $t,\eta,\xi$ as functions of $(x,z)$.  Once again, $t(x,z) = \Theta_g(x,z)$ and since $x=y$ on the critical set the equation $z = d_\xi \varphi(t,x,\xi)$ means that $\xi(x,z)$ has the same role as above.  The equation $\eta = d_y\varphi(t,x,\xi)$ then implies that $\eta(x,z)$ is the covector at $x$ of the unit speed geodesic from $z$ to $x$.

Given the observations, the only possible difference between the leading order terms in the stationary phase expansions result from evaluating $q_\nu$ at the two different sets of critical points.  The first is $q_\nu(d_\xi\varphi(t,x,\xi),\xi(x,z))$, the second is $q_\nu(x,\eta(x,z))$.  We now appeal to the observation following the definition of $Q_\nu$ that $q_\nu$ is invariant under the flow $\chi_t$ at points within the support of $v$.  Thus since $\chi_t(d_\xi\varphi(t,x,\xi),\xi(x,z)) = \chi_t(z,\xi(x,z))=(x,\eta(x,z))$ by the observations above, the proof is now complete.
\end{proof}

\subsubsection{Preliminaries for bilinear estimates and spatial localization} Taking $N$ large in Lemma \ref{L:pdofiocomp}, we have that up to an error term which is a bilinear operator in
$(h_1,h_2)$ satisfying stronger $L^{q/2}$ bounds,
\begin{multline}\label{bilinearkernel}
   \Upsilon_{j,\mu,\tilde{\mu}}\left(h_1,h_2\right)(x) = \\
   \lambda^{d-1} \iint e^{i\lambda(\Theta_g(x,z)+\Theta_g(x,\tilde z))}
   V_{j,\mu}(x,z)V_{j,\tilde \mu}(x,\tilde z)h_1(z)h_2(\tilde z)\,dzd\tilde z + \text{error}.
\end{multline}

Next, we claim that when $V_{j,\mu}(x,z)V_{j,\tilde \mu}(x,\tilde z)\neq 0$,
\begin{equation}\label{covectorsep}
  |d_x \Theta_g(x,z) - d_x \Theta_g(x,\tilde z) |_{g(x)} \approx 2^j.
\end{equation}
To see this, recall $d_x \Theta_g(x,z)$, $d_x \Theta_g(x,\tilde z)$ give the unit covector at $x$ of the unit speed geodesic from $z$ to $x$, $\tilde z$ to $x$ respectively.  Now consider the coordinates $w',\tilde{w}'$ on the $w_1=0$ plane where the geodesics through $(x,z)$ and $(x,\tilde z)$ intersect this plane and let $\omega$, $\tilde{\omega}$ denote the unit covectors of the geodesic at the respective intersection points, see Figure \ref{F:transverse}.  The assumption on $x,z,\tilde z$ ensures that with respect to Euclidean distance determined by the coordinate system, we have $|\omega-\tilde \omega|\approx 2^j$.  Also note that it suffices to show \eqref{covectorsep} with the intrinsic distance replaced by Euclidean.  Since the geodesics can be reparameterized in terms of the first coordinate $x_1$, we have that
\begin{equation*}
  |d_x \Theta_g(x,z) - d_x \Theta_g(x,\tilde z) | \approx |w'-\tilde{w}'| + |\omega-\tilde{\omega}| \approx |w'-\tilde{w}'| + 2^j,
\end{equation*}
as the covectors on the very left here are over the same point. The lower bound in \eqref{covectorsep} now follows.  For the upper bound, integrating the equation for the $w'$ components of the geodesics parameterized by $x_1$ gives
\begin{equation*}
 |w'-\tilde{w}'| \lesssim |x_1| |d_x \Theta_g(x,z) - d_x \Theta_g(x,\tilde z) | \lesssim \epsilon |d_x \Theta_g(x,z) - d_x \Theta_g(x,\tilde z) | .
\end{equation*}
Taking $\epsilon>0$ small in \eqref{epsilonchart}, the rest of \eqref{covectorsep} follows.

\begin{figure}[t]
 {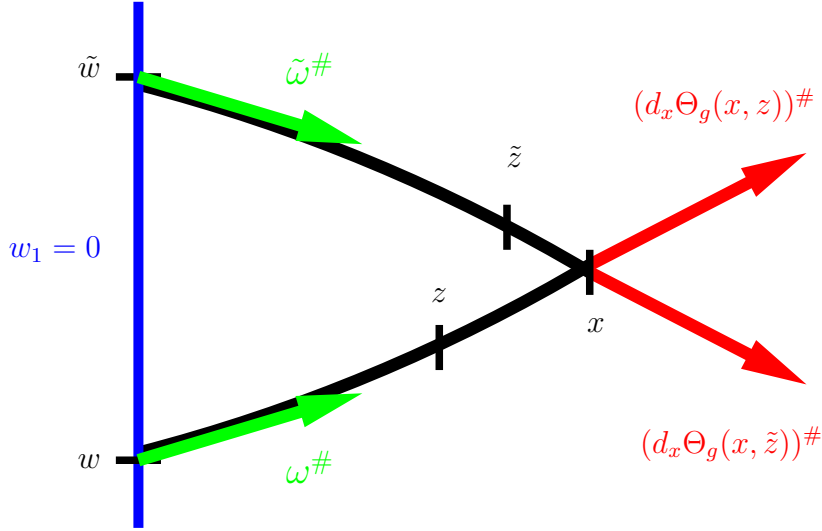}
\caption{Intersecting geodesics at $x$ passing through $z,\tilde z$ and their tangent vectors as raised covectors.}\label{F:transverse}
\end{figure}

Redefine $\beta$ to be a smooth bump function satisfying for $w'\in \RR^{d-1}$,
\begin{equation}\label{newbetaproperties}
  \sum_{l\in 2^j\mathbb{Z}^{d-1}}\beta^2\left(2^{-j}(w'-l)\right)=1, \qquad \supp(\beta)\subset [-1,1]^{d-1}.
\end{equation}
We now return to $\Phi$ as in \eqref{planeinverse} and let $\Upsilon_{j,\mu,\tilde{\mu},l}$ denote the operator defined by replacing the amplitude $V_{j,\mu}(x,z)V_{j,\tilde \mu}(x,\tilde z)$ in \eqref{bilinearkernel} by
\begin{equation}\label{spatialloc}
V_{j,\mu,\tilde \mu, l}(x,z,\tilde z):=  V_{j,\mu}(x,z)V_{j,\tilde \mu}(x,\tilde z)\beta\left(2^{-j}(\Phi(x,\omega(x,\mu))-l)\right),
\end{equation}
where as before, $\omega(x,\mu)$ is the covector of the unit speed geodesic passing through $x$ such that the covector at the intersection point with $y_1=0$ is $\mu/|\mu|_g$ (cf. \eqref{suppqtilde} and the discussion after \eqref{planeinverse}).  Hence this geodesic has coordinates $(\Phi(x,\omega(x,\mu)),\mu)$ at its intersection point with $y_1=0$.  The introduction of the bump function thus has the effect of localizing the amplitude of the oscillatory integral operator to a $2^j$ neighorhood of the geodesic which passes through the $y_1=0$ plane at $(0,l)$ with unit covector $\mu/|\mu|_g$.

We now claim that while \eqref{spatialloc} localizes the kernel in the $x$ coordinates, the support properties of the $V_{j,\mu}$ in Lemma \ref{L:pdofiocomp} imply that
\begin{equation}\label{spatiallocz}
  \text{ if } V_{j,\mu}(x,z)\beta(2^{-j}(\Phi(x,\omega(x,\mu))-l))\neq 0, \;\text{ then } |\Phi(z,\omega(z,\mu))-l| \lesssim 2^j.
\end{equation}
To see this, consider the unit speed geodesic joining $z$ to $x$ and let the unit covectors of the geodesic at these point be denoted by $\tilde{\omega}_z$, $\tilde{\omega}_x$ respectively.  Hence $\Phi(z,\tilde{\omega}_z) = \Phi(x,\tilde{\omega}_x) $ as this geodesic segment will intersect the $y_1=0$ at a unique point.  We now have if the hypothesis in \eqref{spatiallocz} is satisfied, then
\begin{multline*}
  |\Phi(z,\omega(z,\mu))-l| \leq\\
  |\Phi(z,\omega(z,\mu))-\Phi(z,\tilde{\omega}_z)|+|\Phi(x,\tilde{\omega}_x)-\Phi(x,\omega(x,\mu))|+|\Phi(x,\omega(x,\mu))-l|\\
  \lesssim |\omega(z,\mu)-\tilde{\omega}_z| +|\omega(x,\mu)-\tilde{\omega}_x| + 2^j,
\end{multline*}
where we have used the Lipschitz bounds on $\Phi$.  Since $V_{j,\mu}(x,z)\neq 0$, Lemma \ref{L:pdofiocomp} and \eqref{suppqtilde} imply that $(z,\omega_z) \in \supp(\tilde{q}_{j,\mu}(z,\cdot))$ and $|\omega(z,\mu)-\tilde{\omega}_z| \lesssim 2^j$.  Recalling the discussion after \eqref{planeinverse}, $\Psi(x,\cdot),\Psi(z,\cdot)$ are invertible and hence
\begin{equation*}
  |\omega(x,\mu)-\omega_x| \lesssim |\Psi(x,\omega(x,\mu))-\Psi(x,\omega_x)| = |\mu-\Psi(z,\omega_z)| \lesssim |\omega(z,\mu)-\tilde{\omega}_z|,
\end{equation*}
using that $\omega(x,\cdot)$ inverts $\Psi(x,\cdot)$.  The claim in \eqref{spatiallocz} now follows.

We now claim it suffices to prove \eqref{bilinearaoreduction}, with $\Upsilon_{j,\mu,\tilde{\mu}}$ replaced by $\Upsilon_{j,\mu,\tilde{\mu},l}$.  First note that while the identity in \eqref{newbetaproperties} applies to the square sum over the $\beta$, we have more generally
\begin{equation}\label{newbetaapprox}
  \sum_{l\in 2^j\mathbb{Z}^{d-1}}\beta^r\left(2^{-j}(w'-l)\right)\approx_r 1, \qquad 0<r<\infty,
\end{equation}
since for any $l$,
\begin{equation*}
  \#\{\tilde{l} \in 2^j\mathbb{Z}^{d-1}: \supp(\beta(2^{-j}(\cdot -l)))\cap \supp(\beta(2^{-j}(\cdot -\tilde l)))\} = \mathcal{O}(1).
\end{equation*}  Hence H\"older's inequality gives
\begin{equation}\label{littlelq2}
\left\| \Upsilon_{j,\mu,\tilde{\mu}}\left(Q_{j,\mu}h,Q_{j,\tilde{\mu}}h\right)\right\|_{L^{q/2}} \lesssim
\left( \sum_l
\left\| \Upsilon_{j,\mu,\tilde{\mu},l}\left(Q_{j,\mu}h,Q_{j,\tilde{\mu}}h\right)\right\|_{L^{q/2}}^{\frac q2}\right)^{\frac 2q}
\end{equation}
We now apply \eqref{spatiallocz} to see that there exists a bump function $\tilde{\beta}$ such that $\tilde{\beta}_{\mu,l} (z):=\tilde{\beta}(2^{-j}(\Phi(z,\omega(z,\mu))-l))$ satisfies
\begin{equation*}
  \Upsilon_{j,\mu,\tilde{\mu},l}\left(Q_{j,\mu}h,Q_{j,\tilde{\mu}}h\right) = \Upsilon_{j,\mu,\tilde{\mu},l}\left(\tilde{\beta}_{\mu,l} Q_{j,\mu}h,\tilde\beta_{\tilde\mu,l} Q_{j,\tilde{\mu}}h\right)
\end{equation*}
If \eqref{bilinearaoreduction} holds with the additional localization, i.e. this holds with $\Upsilon_{j,\mu,\tilde{\mu},l}$ replacing $\Upsilon_{j,\mu,\tilde{\mu}}$, then the more general bound holds since
\begin{multline*}
  \left( \sum_l \left\| \tilde{\beta}_{\mu,l} Q_{j,\mu}h\right\|_{L^2}^{\frac q2}\left\|\tilde{\beta}_{\tilde\mu,l} Q_{j,\tilde \mu}h\right\|_{L^2}^{\frac q2}\right)^{\frac 2q}\\
  \lesssim   \left( \sum_l \left\| \tilde{\beta}_{\mu,l} Q_{j,\mu}h\right\|_{L^2}^{q}\right)^{\frac 1q} \left(\sum_l \left\|\tilde{\beta}_{\tilde\mu,l} Q_{j,\tilde \mu}h\right\|_{L^2}^{q}\right)^{\frac 1q}
  \lesssim \left\|  Q_{j,\mu}h\right\|_{L^2}\left\|Q_{j,\tilde \mu}h\right\|_{L^2}
  ,
\end{multline*}
where the first inequality here uses Cauchy-Schwarz and the second uses the embedding $\ell^2 \hookrightarrow \ell^q$ along with the fact that we may assume that \eqref{newbetaapprox} holds with $\beta$ replaced by $\tilde\beta$.


\subsubsection{Fermi coordinates, parabolic scaling, and the proof of Lemma \ref{L:offdiag}}
As noted above, the additional localization means that $V_{j,\mu,\tilde \mu, l}$ vanishes unless $x$ is in a $2^j$-neighborhood of the image of the geodesic $\gamma_{l,\mu}$ which passes through the $x_1=0$ plane at $(0,l)$ with unit covector $\mu$.  Given $x$, let $\omega_{l,\mu}(x)$ denote the unit covector on $\gamma_{l,\mu}$ at the closest point to $x$.  The same idea in Lemma \ref{L:pdofiocomp} shows that since $\omega_{l,\mu}(x)$ is within a distance of $\mathcal{O}(2^j)$ to both $\omega(x,\mu)$ and $\omega(x,\tilde \mu)$,
\begin{equation*}
  \left|(\omega_{l,\mu}(x)^\#)^k\prtl^\beta_{x}V_{j,\mu,\tilde \mu, l}(x,z,\tilde z)\right| \lesssim_{k,\beta} 2^{-j|\beta|}.
\end{equation*}

The additional localization of $\Upsilon_{j,\mu,\tilde{\mu},l}$ now allows us to change to Fermi coordinates which straighten $\gamma_{l,\mu}$ so that the $x_1$-coordinate parameterizes the geodesic and $|x'| = \Theta_g(x,\gamma)$. The regularity bounds for $V_{j,\mu,\tilde \mu, l}$ transform as
\begin{equation*}
  \left|\prtl_{x_1}^k\prtl^\beta_{x}V_{j,\mu,\tilde \mu, l}(x,z,\tilde z)\right| \lesssim_{k,\beta} 2^{-j|\beta|}.
\end{equation*}
Note that \eqref{covectorsep} holds in these coordinates as it is intrinsic to $(M,g)$.  Moreover, in these coordinates we may still view the cosphere bundle as a graph in the last $d-1$ variables, and in particular for each $x$, there exists a strictly concave function $\xi'\mapsto r(x,\xi')$ which defines $|\xi|_{g(x)} =1$ in that
\begin{equation*}
  \prtl_{x_1}\Theta_g (x,z) = r(x,d_{x'}\Theta_g(x,z)).
\end{equation*}

For $i=1,2$, let $h_i^{z_1}(z') = h_i(z_1,z')$. After an application of Minkowski's inequality, \eqref{bilinearaoreduction} is reduced to showing that uniformly in $z_1$, $\tilde z_1$ we have
\begin{multline}\label{dilationreduction}
\left\|  \iint e^{i\lambda(\Theta_g(x,z)+\Theta_g(x,\tilde z))}
  V_{j,\mu,\tilde \mu, l}(x,z,\tilde z)h_1^{z_1}(z')h_2^{\tilde z_1}(\tilde z')\,dz'd\tilde z'\right\|_{L^{\frac q2}(\RR^d)}\\
  \lesssim
  \lambda^{- \frac{2d}{q}}2^{j(d-1-\frac{2(d+1)}{q})}\|h_1^{z_1}\|_{L^2(\RR^{d-1})}\|h_2^{\tilde z_1}\|_{L^2(\RR^{d-1})},
\end{multline}
where we have cancelled the factor $\lambda^{d-1}$ from \eqref{bilinearkernel} and the right hand side of \eqref{bilinearaoreduction}. This will follow from the parabolic rescaling $(x',z')\mapsto (2^j x',2^j z')$, rewriting the oscillatory factor as
\begin{equation*}
   e^{i\lambda2^{2j}(\tilde\Theta(x,z)+\tilde\Theta(x,\tilde z))} \quad \text{with} \quad \tilde\Theta(x,z) = 2^{-2j}\Theta_g(x_1,2^jx',z_1,2^jz').
\end{equation*}
The dilation ensures that the derivatives of the amplitude are uniformly bounded in $\lambda$ and $j$.  Moreover, in the new coordinates $\prtl_{x_1}\tilde\Theta = \tilde{r}(x,d_{x'}\tilde\Theta)$ with $\tilde{r}(x,\xi') = 2^{-2j}r(x_1,2^{j}x',2^j\xi')$, whose Hessian in $\xi'$ satisfies the same bounds as $r$. Hence \eqref{dilationreduction} follows from \cite[Theorem 3.3]{BlairSoggeKaknik}, which removes the $\veps$-loss in the bilinear estimates of Lee \cite[Theorem 1.1]{LeeBilinear}.  The latter in turn generalize bilinear Fourier restriction estimates of \cite{TaoBilinear}, \cite{WolffBilinear}.  Indeed, \eqref{covectorsep} ensures that over the support of the dilated amplitude,
\begin{equation*}
 \big| d_{x'}\tilde\Theta(x,z)-d_{x'}\tilde\Theta(x,\tilde z)\big| \approx 1 ,
\end{equation*}
which with the concavity of $\tilde{r}(x,\cdot)$, can be seen to be sufficient for the condition\footnote{In comparison to \cite{BlairSoggeKaknik}, $x$ in the present work plays the role of $z=(x,s)$ there.} in these theorems, as the differentials are uniformly transverse in the graph of $\tilde r$.  When $d=2$, one could also show \eqref{dilationreduction} by the method in \cite[Lemma 3.3]{BlairSoggeRefined} which follows the approach of H\"ormander \cite{HormanderFLP} at the valid endpoint $q=4$ and does not require a dilation of coordinates.

\subsection{Almost orthogonality}\label{SS:aoforbilinear} In this section, we prove Lemma \ref{L:aobound}, primarily focusing on \eqref{aobound}.  The principle is essentially the same as in the proofs of  \cite[(3-4), (3-10)]{BlairSoggeRefined} or \cite[Lemma 6.7]{mss93} and is a ``variable coefficient'' version of the almost orthogonality principle in \cite[Lemma 6.1]{TaoVargasVega}.  The cases $r=1, \infty$ follow from the triangle inequality, so it suffices to consider $r=2$ and interpolate.  Note that when $r=2$, the left hand side of \eqref{aobound} is
\begin{equation*}
  \sum_{(\mu_1,\tilde{\mu}_1),(\mu_2,\tilde{\mu}_2) \in \Xi_j}
  \left\langle \Upsilon_{j,\mu_1,\tilde{\mu}_1}\left(Q_{j,\mu_1}h,Q_{j,\tilde{\mu}_1}h\right),
  \Upsilon_{j,\mu_2,\tilde{\mu}_2}\left(Q_{j,\mu_2}h,Q_{j,\tilde{\mu}_2}h\right)\right\rangle_{L^2}
\end{equation*}
For any fixed $C$, and any $\mu_1$, we have $\#\{\mu_2:|\mu_1-\mu_2| \leq C2^j\} = \mathcal{O}(1)$.  Therefore summing over the pairs $(\mu_1,\tilde{\mu}_1),(\mu_2,\tilde{\mu}_2) \in \Xi_j$ such that $|\mu_1-\mu_2|\leq C2^j$ satisfies the bound
\begin{multline*}
  \sum_{|\mu_1-\mu_2|\leq C2^j}
  \left\langle \Upsilon_{j,\mu_1,\tilde{\mu}_1}\left(Q_{j,\mu_1}h,Q_{j,\tilde{\mu}_1}h\right),
  \Upsilon_{j,\mu_2,\tilde{\mu}_2}\left(Q_{j,\mu_2}h,Q_{j,\tilde{\mu}_2}h\right)\right\rangle_{L^2} \\ \lesssim   \sum_{(\mu,\tilde{\mu})\in \Xi_j} \left\|\Upsilon_{j,\mu,\tilde{\mu}}\left(Q_{j,\mu_1}h,Q_{j,\tilde{\mu}_1}h\right)\right\|_{L^2}^2.
\end{multline*}
It thus suffices to show that there exists $C$ such that the sum over the off-diagonal pairs satisfies
\begin{multline*}
  \sum_{|\mu_1-\mu_2|> C2^j}
  \left\langle \Upsilon_{j,\mu_1,\tilde{\mu}_1}\left(Q_{j,\mu_1}h,Q_{j,\tilde{\mu}_1}h\right),
  \Upsilon_{j,\mu_2,\tilde{\mu}_2}\left(Q_{j,\mu_2}h,Q_{j,\tilde{\mu}_2}h\right)\right\rangle_{L^2} \\ \lesssim \lambda^{-2N}\left\|h\right\|_{L^2}^2.
\end{multline*}

Recalling the form of the kernel of $\Upsilon_{j,\mu,\tilde \mu}$ in \eqref{bilinearkernel}, the main idea is that if $(\mu_1,\tilde\mu_1), (\mu_2,\tilde\mu_2) \in \Xi_j$ and $2^{-j}|\mu_1-\mu_2|$ is sufficiently large then for any $N \in \mathbb{N}$ (possibly larger than that in the previous display),
\begin{equation*}
\begin{gathered}
 \left|    \int e^{i\lambda(\Theta_g(x,z)+\Theta_g(x,\tilde z)-\Theta_g(x,w)-\Theta_g(x,\tilde w))}
   V_{j,\mu_1,\tilde \mu_1, \mu_2,\tilde\mu_2}(x,z,\tilde z, w, \tilde w)dx\right| \lesssim_N \lambda^{-N}, \\
    V_{j,\mu_1,\tilde \mu_1, \mu_2,\tilde\mu_2}(x,z,\tilde z, w, \tilde w):=V_{j,\mu_1}(x,z)V_{j,\tilde \mu_1}(x,\tilde z)V_{j,\mu_2}(x,w)V_{j,\tilde \mu_2}(x,\tilde w).
\end{gathered}
\end{equation*}
Given the regularity estimate \eqref{pdofiocompest} and $2^j \gtrsim \lambda^{-1/8}$, this in turn follows from integration by parts and the bound
\begin{equation*}
\Big| d_x\big(\Theta_g(x,z)+\Theta_g(x,\tilde z)-\Theta_g(x,w)-\Theta_g(x,\tilde w)\big)\Big| \gtrsim 2^j.
\end{equation*}
Indeed, each integration by parts will yield a gain of $\lambda^{-1}2^{2j}$ which is at least $\lambda^{-3/4}$.  Given \eqref{covectorsep}, this in turn will follow from
\begin{equation*}
\Big| d_x\big(\Theta_g(x,z)-\Theta_g(x,w)\big)\Big| \gtrsim 2^j,
\end{equation*}
again assuming $2^{-j}|\mu_1-\mu_2|$ is sufficiently large.  Since $d_x\Theta_g(x,z)$, $d_x\Theta_g(x,w)$ give the covectors along the geodesics joining $z,w$ to $x$, this follows from the same principle as in \eqref{covectorsep}: if $V_{j,\mu_1}(x,z)V_{j,\mu_2}(x,w) \neq 0 $, the two geodesics through $x$ passing through $z,w$ respectively intersect the $y_1=0$ hyperplane with covectors pointing in the direction $\mu_1$, $\mu_2$ respectively.  Since these two vectors are separated by a distance $\gtrsim C2^j$, this is enough.

The bound \eqref{aobounddiag} is shown similarly, the only difference is that we did not multiply the kernel $\tilde{\sigma}_\lambda$ by a localizing factor\footnote{This meant that the application of the linear theory in \eqref{linearapp} was straightforward as the bounds on the amplitude defining $\sigma_\lambda$ and its derivatives are uniform in $\lambda$.} akin to the $\tilde{q}_{\mu,j}$ (cf. \eqref{supportprops}).  However, if we consider the composition $\tilde{\sigma}_\lambda\circ Q_\nu$ as in Lemma \ref{L:pdofiocomp} the proof is nearly identical to the $2^j \approx \lambda^{-1/8}$ case here.

\section{Weaker geometric conditions}
We conclude this work with a discussion of the prospects for proving Theorem \ref{T:mainthm} under weaker hypotheses on the sectional curvatures of $(M,g)$.  Assuming that $(M,g)$ has no conjugate points, the second author showed in \cite{sogge2016localized} that if one had a $o(\lambda^{\delta(p,d)})$ gain in the $L^p$ bounds when $2< p < p_c$, then this would imply a $o(\lambda^{\frac{1}{p_c}})$ gain in the $L^{p_c}$ bounds.  Here we show that there are intermediate hypotheses, stronger than assuming no conjugate points, but weaker than nonpositive curvature, that yield a bound with a logarithmic gain of the form \eqref{mainbound} with a possibly smaller value of $\veps_0$.

There are only two places in the argument above where the nonpositive curvature hypothesis was used in the arguments above over the implicit no conjugate point hypothesis: in \eqref{berardbd} and in the proof of Theorem \ref{T:berardmicro} when bounding the expression of $V(\tilde x, \tilde y)$ in \eqref{mainasymp}, \eqref{coeff}.  In the latter case, the observations in \cite{SoggeZelditchL4} were recalled, showing that the leading coefficient in the Hadamard parametrix $\vartheta(\tilde x, \tilde y)$ is uniformly bounded when the curvatures are nonpositive, yielding \eqref{Ualphabd}.  The other bounds \eqref{asympreg} follow from lower bounds on the curvature and Jacobi field estimates, and here one can allow for the following algebraic growth in $\vartheta$, which is also enough to show \eqref{berardbd},
\begin{equation}\label{conjptgrowth}
\vartheta(\tilde x, \tilde y) \lesssim \Theta_{\tilde g}(\tilde x, \tilde y)^{\frac{d-1}{2}}.
\end{equation}

Recall that $\vartheta$ is characterized by $dV_g = \vartheta^{-2}(\tilde x,\tilde y) d\mathcal{L}(\tilde y)$ in normal coordinates at $\tilde x$, with $\mathcal{L}$ denoting Lebesgue measure.  To motivate the proof of \eqref{conjptgrowth}, let $\tilde \gamma(t):\RR \to \tilde M$ be the unit speed geodesic joining $\tilde x$ to $\tilde y$ with $\tilde \gamma(0) = \tilde x$.  Let $E_1, \dots, E_{d-1} \in T_{\tilde x} \tilde M$ be an orthornormal basis for the orthogonal complement of $\dot{\tilde{\gamma}}(0)$.  Then let $Y_j(t)$ be the normal Jacobi field along $\tilde \gamma(t)$ with initial condition $Y_j(0) =0$ and covariant derivative $D_t Y_j(0) = E_j$.  Also let $Z_j(t)$ denote parallel translation of $E_j$ along $\tilde \gamma(t)$.  The fields $Y_1,\dots, Y_{d-1}$ thus determine a fundamental matrix for solutions $Y(t)$ to the Jacobi equation along $\tilde \gamma(t)$ with $Y(0)=0$ and $\dot{\tilde{\gamma}}(0)\perp Y(t)$; we denote this linear transformation as $\mathcal{Y}(t)$.  Taking polar coordinates it can be seen that
\begin{equation}\label{thetajacobi}
\vartheta^{-2}(\tilde x,\tilde y) = t^{1-d}\left| \det (Y_1(t),\dots,Y_{d-1}(t))\right|, \qquad t=\Theta_{\tilde g} (\tilde x, \tilde y),
\end{equation}
though the right hand side of this identity is more accurately taken to be the determinant of the change of basis matrix from $Y_1,\dots,Y_{d-1}$ to $Z_1,\dots, Z_{d-1}$.

The aforementioned observations in \cite{SoggeZelditchL4} (G\"unther comparison theorem), use that if $M$ has nonpositive sectional curvatures, comparison theorems for Jacobi fields imply that $|Y_j(t)|_g \geq t$, and hence after other considerations, $\vartheta$ is in fact uniformly bounded.  Otherwise, when $(M,g)$ merely lacks conjugate points, the proof of \eqref{conjptgrowth} uses that the determinant in \eqref{thetajacobi} is uniformly bounded from below.  When $d=2$, it was observed in \cite{Berard} that such lower bounds follow from results of Green \cite{GreenConjugatePoints}.  When $d\geq 3$, these were formalized in \cite{bonthonneau2016lower}, though it seems that the crucial bounds were known to Freire and Ma\~n\'e \cite[Lemma I.3]{FreireMane}.

However, \eqref{conjptgrowth} by itself is not enough to imply the kernel estimates \eqref{berardmicro} in Theorem \ref{T:berardmicro}.  Indeed, any reasonable substitute for the summation argument \eqref{stabsum} would require that
\begin{equation}\label{coeffgrowthacceptable}
\vartheta(\tilde x, \tilde y) \lesssim   \Theta_{\tilde g} (\tilde x, \tilde y)^{\frac{d-1}{2}(1-\delta_0)} \text{ for some } 0< \delta_0 <1,
\end{equation}
at which point one would have (with $T=c_0\log \lambda$ as throughout)
\begin{equation}\label{stabsumwocp}
\sum_{\alpha \in \Gamma_{\tuber}}|U_{\alpha}(\tilde w, \tilde z)|\lesssim
\begin{cases}
\lambda^{\frac{d-1}{2}}T^{-\min(1,\frac{(d-1)\delta_0}{2})}, & \delta_0 \neq \frac{2}{d-1},\\
\lambda^{\frac{d-1}{2}}(\log T)T^{-1}, & \delta_0 = \frac{2}{d-1}.
\end{cases}
\end{equation}
The same considerations in \S\ref{S:reviewandmicro} would then hold, though the definition of $c(\lambda)$ in \eqref{berardmicro} would have to be adjusted to be consistent with right hand side here. This in turn leads to adjustments in the exponents of $\log \lambda$ throughout \S\ref{S:reviewandmicro} and \S\ref{S:proofbycontra} which are not difficult to compute.  The remaining considerations in Remark \ref{R:optimalexponent} would then determine the exponent $\veps_0$.

We have outlined the proof of:
\begin{theorem}\label{T:wocp}
Suppose $M$ has no conjugate points and that \eqref{coeffgrowthacceptable} is satisfied.  In particular, this holds if along any unit speed geodesic $\tilde \gamma:\RR \to \tilde M$, the linear transformation $\mathcal{Y}$ determined by the $Y_1, \dots, Y_{d-1}$ above satisfies
\begin{equation}\label{jacobihyp}
 |\mathcal{Y}(t)X|_{g(\tilde\gamma(t))} \gtrsim t^{\delta_0}|X|_{g(\tilde\gamma(0))} , \qquad t \gtrsim 1, \;X \in T_{\tilde{\gamma}(0)}M,\; X \perp \dot{\tilde\gamma}(0).
\end{equation}
for some implicit constant depending only on $(\tilde M, \tilde g)$. Then there exists $\veps_0>0$, possibly different from that in Remark \ref{R:optimalexponentintro}, such that \eqref{mainbound} holds as in the conclusion in Theorem \ref{T:mainthm}.
\end{theorem}

\begin{remark}
In general, the validity of \eqref{coeffgrowthacceptable} and \eqref{jacobihyp} on an arbitrary manifold without conjugate points appears to be a long standing open problem.  As observed above, they are satisfied with $\delta_0=1$ when the sectional curvatures are nonpositive.  An argument of Berger \cite[\S3]{BergerVolume}, shows that \eqref{coeffgrowthacceptable} holds with $\delta_0=1$ if one has the intermediate hypothesis that $M$ has no focal points\footnote{While the results in Berger's work are phrased in terms of the convexity radius, the argument fundamentally uses the no focal points hypothesis that for any nontrivial normal Jacobi field with $Y(0)=0$ as above, $|Y(t)|_g$ is an increasing function of $t$.}.  In particular, we could have stated our original theorem for Riemannian manifolds without focal points instead of those with nonpositive curvatures, with exactly the same exponent $\veps_0$.  Results of Klingenberg \cite{KlingenbergAnosov} and Ma\~n\'e \cite{ManeKlingenbergTheorem} show that when the geodesic flow is Anosov, then in fact exponentially growing lower bounds in \eqref{jacobihyp} are satisfied, yield the same exponent when $d \geq 4$ and even better ones for $d=2,3$ as in Remark \ref{R:negcurvature}.  Finally, Eschenburg \cite[Proposition 6]{EschenburgHorospheres} showed that \eqref{jacobihyp} holds with $\delta_0 = \frac 12$ for manifolds with a so-called ``$\rho$-bounded asymptote'' condition, which is stronger than assuming there are no conjugate points, but weaker than assuming there are no focal points.
\end{remark}

\begin{remark}
The observations here carry over similarly to the authors' work \cite{BlairSoggeToponogov} where the key issue was to obtain a bound of the form \eqref{stabsumwocp} for a slightly different choice of $Q_\lambda$ and use Lemma \ref{L:Toponogov} similarly to achieve bounds analogous to \eqref{mainmicro}.  Consequently, if \eqref{coeffgrowthacceptable} is satisfied, we have for $\|f\|_{L^2(M)}=1$ and $\tubel$ as the tubular neighborhood of diameter $\lambda^{-1/2}$ about a geodesic segment $\gamma$ in $M$:
\begin{equation}\label{kaknikaverage}
\int_{\tubel}|\mathbf{1}_{[\lambda, \lambda(\log\lambda)^{-1}]}(P)f|^2\,dV_g \lesssim
\begin{cases}
 (\log\lambda)^{-\min(\frac{d-1}{2}\delta_0,1)}, & \delta_0 \neq \frac{2}{d-1},\\
 (\log\lambda)^{-1}\log \log \lambda, & \delta_0 = \frac{2}{d-1}.
\end{cases}
\end{equation}
In particular, the results in \cite{BlairSoggeToponogov} hold for $(M,g)$ without focal points. Given the results in \cite{BlairSoggeRefinedHighDimension}, these considerations in turn yields a logarithmic gain in the known $L^p(M)$ bounds on spectral clusters when $2<p<p_c$ (possibly with a larger exponent of $(\log\lambda)^{-1}$ than what would result from interpolating the main bound \eqref{mainbound} with the trivial $p=2$ bounds).
\end{remark}

\bibliographystyle{amsalpha}
\bibliography{bibtexdata}
\end{document}